\documentclass[11pt]{amsart}
\usepackage{amsmath}
\usepackage{amssymb}
\usepackage{amsthm}
\usepackage{enumerate}
\usepackage{url}
\usepackage{stmaryrd}
\usepackage{color}
\usepackage{tikz}
\usepackage{listings}
\usepackage[margin=2.5cm]{geometry}
\usepackage{courier}

\lstdefinelanguage{GAP}{%
    morekeywords=[2]{and,break,continue,do,elif,else,end,fail,false,fi,for,%
        function,if,in,local,mod,not,od,or,rec,repeat,return,then,true,%
        until,while,%
        List, Product, Length, First, IsZero, Remove, Add, Difference, LeadingCoefficient, FactorizationsIntegerWRTList, ShallowCopy, DegreeOfLaurentPolynomial, Length, NumericalSemigroup, AmbientNumericalSemigroupOfIdeal, MinimalGenerators, Union, IntersectionIdealsOfNumericalSemigroup},%
    moredelim=[s][\color{blue}]{gap}{>},%
    moredelim=[s][\color{red}]{brk}{>},%
    sensitive=true,%
    morecomment=[l]\#,%
    morestring=[b]',%
    morestring=[b]",%
    }%

\title{Canonical bases of modules over one dimensional $\KK$-algebras}

\author{A. Abbas}
\address{Universit\'e d'Angers, Math\'ematiques,
49045 Angers ceded 01, France}
\email{abbas@univ-angers.fr}
\author{A. Assi} 
\address{Universit\'e d'Angers, Math\'ematiques,
49045 Angers ceded 01, France}
\email{assi@univ-angers.fr} 

\author{P. A. Garc\'{\i}a-S\'anchez}
\address{Departamento de \'Algebra and IEMath-GR, Universidad de Granada, E-18071 Granada, Espa\~na}
\thanks{The third author is supported by the projects MTM2014-55367-P, FQM-343, and FEDER funds}

\email{pedro@ugr.es}
\urladdr{www.ugr.es/local/pedro}

\date{\today}

\newtheorem{teorema}{Theorem}[section]

\newtheorem{ex}[teorema]{Example}
\newtheorem{alg}[teorema]{Algorithm}
\newtheorem{p}[teorema]{Theorem}
\newtheorem{proposicion}[teorema]{Proposition}
\newtheorem{lema}[teorema]{Lemma}
\newtheorem{definicion}[teorema]{Definition}

\newtheorem{corolario}[teorema]{Corollary}
\theoremstyle{remark}
\newtheorem{nota}[teorema]{Remark}

\newcommand{\BB}{{\bf B}}
\newcommand{\KK}{{\bf K}}

\newcommand{\NN}{{\mathbb N}}

\DeclareMathOperator{\supp}{supp}
\DeclareMathOperator{\App}{App}

\begin{document}
\maketitle

\begin{abstract}
{Let $\mathbf K$ be a field and denote by $\mathbf K[t]$, the polynomial ring with coefficients in $\mathbf K$. Set $\mathbf A=\mathbf K[f_1,\ldots,f_s]$, with $f_1,\ldots, f_s \in \mathbf K[t]$. We give a procedure to calculate the monoid of degrees of the $\mathbf K$ algebra $\mathbf M=F_1\mathbf A+\cdots +F_r\mathbf A$ with $F_1,\ldots , F_r\in \mathbf K[t]$. We show some applications to the problem  of the classification of plane polynomial curves (that is, plane algebraic curves parametrized by polynomials) with respect to some oh their invariants, using the module of K\"ahler differentials.}
\end{abstract}

\section{Introduction}

Let $\mathbf K$ be a field of characteristic zero and let $f(X,Y)$ be a nonzero irreducible element of ${\mathbf K}[X,Y]$. Let $C=\lbrace (x,y)\in{\mathbf K}^2 \mid f(x,y)=0\rbrace$ be the plane algebraic curve defined by $f$. There are some important invariants that can be associated with $C$: the Milnor number, $\mu(f)$, which is the rank of ${\mathbf K}[X,Y]/(f_X,f_Y)$, and the Turina number, $\nu(f)$, which is the rank of ${\mathbf K}[X,Y]/(f,f_X,f_Y)$ (where $f_X,f_Y$ denote the partial derivatives of $f$). The first one tells us how singular is the family of curves $C_{\lambda}=\lbrace (x,y) \mid f(x,y)-\lambda=0\rbrace$, and the second one tells us how singular is the curve $C$. Suppose that $C$ is parametrized by two polynomials $X(t), Y(t)\in{\mathbf K}[t]$. In this case, we can associate to $f$ a semigroup, denoted $\Gamma(f)$ and defined by $\Gamma(f)=\lbrace \mathrm{d}(g(X(t),Y(t)) \mid g(X,Y)\in {\mathbf K}[X,Y]\setminus (f)\rbrace$, where $\mathrm{d}(g(X(t),Y(t))$ denotes the degree in $t$ of $g(X(t),Y(t))$. 

Let ${\mathbf A}={\mathbf K}[X(t),Y(t)]$ be the ${\mathbf K}$-algebra generated by $X(t),Y(t)$. Then $A$ is the ring of coordinates of $C$. If $\lambda ({\mathbf K}[t]/{\mathbf A})<+\infty$, then $\Gamma(f)$ is a numerical semigroup, and $\mu(f)$ coincides with the conductor of $\Gamma(f)$. Let ${\mathbf M}=X'(t){\mathbf A}+Y'(t){\mathbf A}$ be the ${\mathbf A}$-module generated by the derivatives of $X(t),Y(t)$. The set of degrees in $t$ of elements of ${\mathbf M}$, denoted $\mathrm{d}(\mathbf M)$, defines an ideal of $\Gamma(f)$, and from the definition it follows that for all $s\in \Gamma(f)$, the element $s-1$ is  in $\mathrm{d}(\mathbf M)$. Such an element is called exact. In general, $\mathrm{d}(\mathbf M)$ contains elements that are non exact, and the cardinality of the set of these elements is bounded by the genus of $\Gamma(f)$. Furthermore, this cardinality is nothing but the difference $\mu-\nu$. Hence the numerical semigroup $\Gamma(f)$ and the ideal $\mathrm{d}(\mathbf M)$ offer a good computational approach to the study of these invariants.

This paper has two main goals. Given a ${\mathbf K}$-algebra ${\mathbf A}={\mathbf K}[f_1(t),\ldots,f_s(t)]$, we first describe an algorithm that computes a system of generators of the ideal consisting of degrees in $t$ of elements of the module ${\mathbf M}=F_1(t){\mathbf A}+\dots+F_r(t){\mathbf A}$ (where $f_1(t),\ldots,f_s(t),F_1(t),\ldots,F_r(t)\in{\mathbf K}[t]$). This algorithm uses the one given in [5] in order to compute the semigroup consisting of degrees in $t$ of elements of ${\mathbf A}$. Then we consider the case where ${\mathbf A}={\mathbf K}[X(t),Y(t)]$ is the ring of coordinates of the algebraic plane curve parametrized by $X(t),Y(t)$, and ${\mathbf K}$ is an algebraically closed field of characteristic zero. It turns out that the curve has one place at infinity, and if $f(X,Y)$ is a generator of the curve in ${\mathbf K}[X,Y]$, then the semigroup $\Gamma(f)$ introduced above, which  is the same as the semigroup associated with ${\mathbf A}$, can be calculated from the Abhyankar-Moh theory (see [4]). Using this fact and some techniques introduced in Section 6, we characterize the semigroup of polynomial curves when $\mu-\nu\in\{0,1,2\}$. 

\section{Numerical semigroups and ideals}

\subsection{Numerical semigroups}

Let $S$ be a subset of ${\mathbb N}$. The set $S$ is a submonoid of  ${\mathbb N}$ if the following holds:
\begin{enumerate}

\item $0\in S$,

\item  If $a,b\in S$ then $a+b\in S$.
\end{enumerate}

Clearly, $\lbrace 0\rbrace$ and ${\mathbb N}$ are submonoids of ${\mathbb N}$. Also, if $S$ contains a nonzero element $a$, then $d a\in S$ for all $d\in{\mathbb N}$, and in particular, $S$ is an infinite set. 

Let $S$ be a submonoid of ${\mathbb N}$ and let $G$ be the subgroup of ${\mathbb Z}$ generated by $S$ (that is, $G=\lbrace \sum_{i=1}^s\lambda_i a_i \mid \lambda_i\in{\mathbb Z}, a_i\in S\rbrace$). If $1\in G$, then we say that $S$ is a \emph{numerical semigroup}. This is equivalent to the condition that $\NN\setminus S$ is a finite set.

We set $\mathrm G(S)=\NN\setminus S$ and we call it the set of \emph{gaps} of $S$. We denote by $\mathrm g(S)$ the cardinality of $\mathrm G(S)$, and we call $\mathrm g(S)$ the \emph{genus} of $S$. We set $\mathrm{F}(S)=\max(\mathrm G(S))$ and we call it the \emph{Frobenius number} of $S$. We also define $\mathrm C(S)=\mathrm{F}(S)+1$ and we call it the \emph{conductor} of $S$. The least positive integer of $S$,  $\mathrm m(S)=\inf(S\setminus \lbrace 0\rbrace$ is known as the \emph{multiplicity} of $S$.

Even though any numerical semigroup has infinitely many elements, it can be described by means of finitely many of them. The rest can be obtained as linear combinations with nonnegative integer coefficients from these finitely many. 

Let $S$ be a numerical semigroup and let $A\subseteq S$. We say that $S$ is generated by $A$ and we write $S=\langle A\rangle$ if for all $s\in S$, there exist $a_1,\dots,a_r\in A$ and $\lambda_1,\dots,\lambda_r\in\NN$ such that $a=\sum_{i=1}^r\lambda_ia_i$. Every numerical semigroup $S$ is finitely generated, that is, $S=\langle A\rangle$ with $A\subseteq S$ and $A$ is a finite set.

Let $n\in S^*$. We define the \emph{Ap\'ery set} of $S$ with respect to $n$, denoted $\mathrm{Ap}(S,n)$, to be the set 
\[ \mathrm{Ap}(S,n)=\lbrace s\in S\mid s-n\notin S\rbrace.\]

Let $S$ be a numerical semigroup and let $n\in S^*$. For all $i\in \{1,\ldots, n\}$, let $w(i)$ be the smallest element of $S$ such that $w(i)\equiv i \bmod n$. Then
\[\mathrm{Ap}(S,n)=\lbrace 0,w(1),\dots,w(n-1)\rbrace.\]
Furthermore, $S=\langle n, w(1),\ldots,w(n-1)\rangle$. 

We will be interested in a special class of numerical semigroups, namely free numerical semigroups. The definition is as follows.

\begin{definicion} Let $S=\langle r_0,r_1,\ldots,r_h\rangle$ be a numerical semigroup, and let $d_{i+1}=\gcd(r_0,r_1,\ldots,r_i)$ for all $i\in\lbrace 0,\ldots,h\rbrace$ (in particular $d_1=r_0$ and $d_{h+1}=1$) and $e_i=\frac{d_i}{d_{i+1}}$ for all $i\in\lbrace 1,\ldots,h\rbrace$. We say that $S$ is \emph{free} for the arrangement $(r_0,\ldots,r_h)$ if the following conditions hold:

\begin{enumerate}[(1)]
\item $d_1>d_2>\dots>d_{h+1}=1$,
\item $e_ir_i\in \langle r_0,\ldots,r_{i-1}\rangle$ for all $i\in\lbrace 1,\ldots,h\rbrace$.
\end{enumerate}

\end{definicion}

Note that the notion of freeness depends on the arrangement of the generators. For example, $S=\langle 4,6,13\rangle$ is free for the arrangement $(4,6,13)$ but it is not free for the arrangement $(13,4,6)$. Note also that if $S=\langle r_0,r_1,\ldots,r_h\rangle$ is free with respect to the arrangement  $(r_0,\ldots,r_h)$, then $\mathrm g(S)=\frac{\mathrm C(S)}{2}$.

Let $S=\langle r_0,r_1,\ldots,r_h\rangle$ and suppose that $S$ is free with respect to the arranegment  $(r_0,\ldots,r_h)$. Let the notations be as above, given $s\in{\mathbb Z}$, there exist $\lambda_0,\lambda_1,\ldots,\lambda_h\in \mathbb{Z}$ such that 
\[s=\sum_{i=0}^h \lambda_ir_i \text{ and }0\leq \lambda_i< e_i,\ i\in\{1,\ldots, h\}.\] 
Such a representation is unique. We call it the \emph{standard representation} of $s$. We have $s\in S$ if and only if $\lambda_0\geq 0$. Also any free semigroup has the following properties.

\begin{proposicion}
Let $S=\langle r_0,\ldots, r_h\rangle$ be a free numerical semigroup with respect to the arrangement $(r_0,\ldots, r_h)$. 
\begin{enumerate}[i)]
\item $\mathrm{F}(S)=\sum_{i=1}^h(e_i-1)r_i-r_0$.
\item For all $a,b\in{\mathbb Z}$, if $a+b=\mathrm{F}(S)$, then $a\in S$ if and only if $b\notin S$. In other words, $S$ is a symmetric numerical semigroup.
\item $\mathrm{Ap}(S,r_0)=\left\lbrace \sum_{i=1}^h\lambda_ir_i \mid 0\leq \lambda_i<e_i \text{ for all } i\in\lbrace 1,\ldots,h\rbrace\right\rbrace$. 
\end{enumerate} 
\end{proposicion}  

\subsection{Ideals of numerical semigroups}

Let $S$ be a numerical semigroup of ${\mathbb N}$ and let $I$ be a nonempty set of $\mathbb{N}$. We say that $I$ is a \emph{relative ideal} of $S$ if for all $(a,s)\in I\times S$, $a+S\in I$ ($I+S\subseteq I$ for short) and there exists $d\in\mathbb Z$ such that $d+I\subseteq S$. This second condition is equivalent to saying that $I$ has a minimum.

Define the following order on ${\mathbb Z}: n_1\leq_S n_2$ if  $n_2-n_1\in S$. Let $E\subset {\mathbb N}$. We say that $n\in E$ is a minimal element of $E$ with respect to $\leq_S$ if for all $s\in E$, the condition $s\leq_S n$ implies $n=s$. We denote by 
$\mathrm{Minimals}_{\le_S}(E)$ the set of minimal elements of $E$ with respect to $\le_S$.

If $I$ is an ideal of $S$, then there exist a set $\lbrace a_1,\ldots,a_l\rbrace\subseteq I$ such that $I=\bigcup_{i=1}^l(a_i+S)$. We say that $\lbrace a_1,\ldots,a_l\rbrace$ is a \emph{system of generators}  of $I$. If furthermore $a_k\notin\bigcup_{i\not=k}(a_i+S)$, then we say that $\lbrace a_1,\ldots,a_l\rbrace$ is a \emph{minimal set of generators} of $I$. Observe that all minimal generators are incongruent modulo $\mathrm m(S)$, and thus a minimal set of generators of $I$ has at most $m(S)$ elements. This set coincides with $\mathrm{Minimals}_{\le_S}(I)$.

Intersection of two relative ideals is again a relative ideal. In particular, given $a,b\in\mathbb N$, $(a+S)\cap(b+S)$ is a relative ideal. Assume that $\lbrace a_1,\ldots,a_r\rbrace$ is the set of minimal generators of $(a+S)\cap(b+S)$. We set 
\[\mathrm R(a,b)=\lbrace (a_k-a,a_k-b), k\in\{1,\ldots,r\}\rbrace.\]

\begin{ex} Let $S=\langle 3,4\rangle=\lbrace 0,3,4,6,7,\to\rbrace$ and let $a=3$, $b=5$. We have $3+S=\lbrace 3,6,7,9,10,\to\rbrace$ and $5+S=\lbrace 5,8,9,11,12,\to\rbrace$. Hence $(3+S)\cap(5+S)=\lbrace 9, 11,12,\to \rbrace=(9+S)\cup (11+S)$. Note that $\lbrace 9,11\rbrace$ is the set of minimal elements of $(3+S)\cap(5+S)$ with respect to $\leq_S$ and that $\mathrm R(3,5)=\lbrace (6,4),(8,6)\rbrace$.
\end{ex}

\section{Relators for monomial subalgebras}

Let $S=\langle s_1,\ldots,s_n\rangle$ be a numerical semigroup and let $I$ be a relative ideal of $S$. Let $\lbrace a_1,\ldots,a_r\rbrace$ be a minimal system of generators of $I$. 
Let $\KK$ be a field and consider the algebra ${\bf A}=\KK[t^{s_1},\ldots,t^{s_n}]=\mathbf K[S]$. Let ${\bf M}=t^{a_1}{\bf A}+\dots+t^{a_r}{\bf A}$ and let
\[
\phi: \mathbf A^r\to \textbf M,\quad \phi(f_1,\ldots,f_r)=t^{a_1}f_1+\dots+t^{a_r}f_r.
\] 
The kernel $\ker(\phi)$ is a submodule of $\mathbf A^{r}$. The following result gives explicitely a generating system for $\ker(\phi)$.

\begin{teorema} \label{th:kernel} 
Let $S$
be a numerical semigroup and let $I$ be a relative ideal of $S$ minimally generated by $\{a_1,\ldots,a_r\}$. 
Let $\varphi$ be the morphism 
\[
\phi: \mathbf A^r\to t^{a_1}{\bf A}+\dots+t^{a_r}{\bf A},\quad \phi(f_1,\ldots,f_r)=t^{a_1}f_1+\dots+t^{a_r}f_r.
\]
Then $\ker(\phi)$
is generated by 
\[\left\lbrace t^{\alpha}{\bf e}_i-t^{\beta}{\bf e}_j\mid i,j\in \{1,\ldots, r\},  i\not=j,  (\alpha,\beta)\in \mathrm R(a_i,a_j) \right\rbrace,\] where $\lbrace {\bf e}_1,\ldots,{\bf e}_r\rbrace$ denotes the canonical basis of ${\bf A}^r$\negthinspace. 
\end{teorema}
\begin{proof}
Let $B=\left\lbrace t^{\alpha}{\bf e}_i-t^{\beta}{\bf e}_j\mid i,j\in \{1,\ldots, r\},  i\not=j,  (\alpha,\beta)\in \mathrm R(a_i,a_j) \right\rbrace$. Clearly, $B\subset \ker(\varphi)$. 

Let $\mathbf f= (f_1,\ldots,f_r)\in \ker(\phi)$. We have $\sum_{i=1}^rt^{a_i}f_i=0$. Let $d_i$ be the degree of $f_i$, and assume that $c_it^{d_i}$ is the leading term of $f_i$, $i\in\{1,\ldots,r\}$. As $\sum_{i=1}^rt^{a_i}f_i=0$, there must be $i\in \{2,\ldots,r\}$ and a monomial $kt^s$ of $f_i$ such that $a_1+d_1=a_i+s$ ($s\in S$). Without loss of generality, we may think that $i=2$. 
Thus $a_1+d_1=a_2+s\in (a_1+S)\cap (a_2+S)$, whence $a_1+d_1=a_2+s=\gamma+ s_{12}$ with $\gamma$ a minimal generator of $(a_1+S)\cap (a_2+S)$ and $s_{12}\in S$. Hence $(d_1,s)=(\gamma-a_1+s_{12},\gamma-a_2+s_{12})$. Set $(\alpha,\beta)=(\gamma-a_1,\gamma-a_2)$. Then $(\alpha,\beta)\in \mathrm R(a_1,a_2)$ and $(d_1,s)=(\alpha+s_{12}, \beta+s_{12})$-

We can write $\mathbf f= (f_1,\ldots, f_r)=c_1t^{s_{12}}(t^\alpha \mathbf e_1-t^\beta \mathbf e_2) + \mathbf f'$, with $\mathbf f'= (f_1',f_2',f_3,\ldots,f_r)$, $f_1'=f_1-c_1t^{d_1}$ and $f_2'=f_2+c_1t^{d_1}$. In this way, we have killed the leading term of $f_1$, and $\mathbf f'$ is again in $\ker(\varphi)$. We continue with $\mathbf f'$ until the first component is zero. After that we focus on the second component and so on. We will end up with an expression of the form $\mathbf f^{(n)}=(0,\ldots,0,f^{(n)}_r)\in \ker(\varphi)$. But this leads to $f^{(n)}_r=0$, since otherwise $t^{a_r}f^{(n)}_r$ would not be zero. This concludes the proof.
%
%
%
\end{proof}

\begin{ex}\label{ex:ker} Let $S=\langle 3,4\rangle$ and let $I=(3+S)\cup(5+S)$. 
Let \[\phi: \mathbf A^2\to t^3\KK[t^3,t^4]+t^5\KK[t^3,t^4],\ \phi(f_1,f_2)=t^3f_1+t^5f_2.\] Then $\ker(\phi)$ is generated by $\{(t^6,-t^4),(t^8,-t^6)\}$.
\end{ex}

In light of Theorem \ref{th:kernel}, we can use the following code in \texttt{GAP} (by using the \texttt{numericalsgps} package) to calculate the kernel of $\varphi$. 
\begin{lstlisting}[basicstyle=\ttfamily, tabsize=1, caption={\texttt{R} and \texttt{ker} functions},label=r-ker-code]
R:=function(a,b,s)
	local i, mg;
	i:=IntersectionIdealsOfNumericalSemigroup(a+s,b+s);
	mg:=MinimalGenerators(i);
	return List(mg, m->[m-a,m-b]);
end;

ker:=function(I)
	local mg, s, r, i, j, n;
	s:=AmbientNumericalSemigroupOfIdeal(I);
	mg:=MinimalGenerators(I);
	r:=[];
	n:=Length(mg);
	for i in [1..n] do
		for j in [i+1..n] do
			r:=Union(r,R(mg[i],mg[j],s));
		od;
	od;
	return r;
end;
\end{lstlisting}

Example \ref{ex:ker}, can be calculated as follows.

\begin{verbatim}
gap> s:=NumericalSemigroup(3,4);
<Numerical semigroup with 2 generators>
gap> I:=[3,5]+s;
<Ideal of numerical semigroup>
gap> ker(I);
[ [ 6, 4 ], [ 8, 6 ] ]
\end{verbatim}

\section{Basis of a $\KK$-algebra}\label{bases of algebras}

Let $\KK$ be a field and let $f_1(t),\ldots,f_s(t)\in \KK[t]$. Let $\mathbf A =\KK[f_1,\ldots,f_s]$, which is a subalgebra of $\KK[ t]$. Assume, without loss of generality, that $f_i$ is monic for all $i\in\{1,\ldots,s\}$. Given $f(t)=\sum_{i=0}^pc_it^i\in \mathbf A $, with $c_p\neq 0$, we set $\mathrm{d}(f)=p$ and $\mathrm M(f)=c_pt^p$, the \emph{degree} and \emph{leading monomial} of $f$, respectively. We also define $\mathrm{supp}(f)=\lbrace i \mid c_i\not=0\rbrace$, the \emph{support} of $f$. 

The set $\mathrm{d}(A)=\lbrace \mathrm{d}(f) \mid f\in \mathbf A \rbrace$ is a submonoid of $\NN$. We shall  assume that $\lambda_\mathbf A ({\KK[t]}/\mathbf A )<\infty$. This implies that $\mathrm{d}(A)$ is a numerical semigroup. 
We say that $\lbrace f_1,\ldots,f_s\rbrace$ is a \emph{basis} of $A$ if $\{\mathrm{d}(f_1),\ldots,\mathrm{d}(f_s)\}$ generates $\mathrm{d}(\mathbf A)$. Clearly, $\lbrace f_1,\ldots,f_s\rbrace$ is a basis of $A$ if and only if $\KK[\mathrm M(f) \mid f\in \mathbf A ]=\KK[\mathrm M(f_1),\ldots,\mathrm M(f_s)]$.

\begin{proposicion}\label{remainder} Given $f(t)\in \KK[t]$, there exist $g(t)\in \mathbf A $ and $r(t)\in \KK[t]$ such that the following conditions hold:

\begin{enumerate}

\item $f(t)=g(t)+r(t)$,

\item if $g(t)\not=0$ (respectively $r(t)\not=0$), then $\mathrm{d}(g)\leq \mathrm{d}(f)$ (respectively $\mathrm{d}(r)\leq \mathrm{d}(f)$),

\item If $r(t)\not=0$, then $\mathrm{supp}(r(t))\subseteq \NN\setminus \langle \mathrm{d}(f_1),\ldots,\mathrm{d}(f_s) \rangle$.

\end{enumerate}
\end{proposicion}
\begin{proof} The assertion is clear if $f\in\KK$. Suppose that $f\notin\KK$ and let $f(t)=\sum_{i=0}^pc_it^i$ with $p=\mathrm{d}(f)>0$.

\begin{enumerate}
\item If $p\notin \langle \mathrm{d}(f_1),\ldots,\mathrm{d}(f_s) \rangle$, then we set $g^1=0, r^1=c_pt^p$ and $f^1=f-c_pt^p$.

\item If $p\in \langle \mathrm{d}(f_1),\ldots,\mathrm{d}(f_s) \rangle$, then  $t^p=\mathrm M(f_1)^{\theta_1}\cdots \mathrm M(f_s)^{\theta_s}$, for some $(\theta_1,\ldots,\theta_s)\in{\mathbb N}^s$ (this expression is not necessarily unique). We set
$g^1=c_pf_1^{\theta_1}\cdots  f_s^{\theta_s}$, $r^1=0$ and
$f^1=f-g^1$.
\end{enumerate}
With this choice of $g^1$ and $r^1$, we have $f=f^1+g^1+r^1$, $g^1\in\bf A$ $r^1\in\mathbf K[t]$, and the following conditions hold:

\begin{enumerate}

\item If $r^1\not=0$, then  $\mathrm{supp}(r^1)\subseteq \NN\setminus \langle \mathrm{d}(f_1),\ldots,\mathrm{d}(f_s) \rangle$.
\item  If $f^1\notin \KK$, then $\mathrm{d}(f^1)<\mathrm{d}(f)=p$. 
\end{enumerate}
Then we restart with $f^1$. Clearly there is $k\geq 1$ such that $f^{k}\in \KK$. We set
$g=g^1+\dots+g^k+f^k$ and $r=r^1+\cdots+r^k$. 
\end{proof}

We denote the polynomial $r(t)$ of Proposition \ref{remainder}  by $\mathrm R(f,\lbrace f_1,\ldots,f_s\rbrace)$. Note that this polynomial is not unique.

\begin{proposicion} 
The set $\lbrace f_1,\ldots,f_s\rbrace$ is a basis of ${\bf A}$ if and only if $\mathrm R(f,\lbrace f_1,\ldots,f_s\rbrace)=0$ for all $f\in {\bf A} $.
\end{proposicion}
\begin{proof} 
Suppose that $\lbrace f_1,\ldots,f_s\rbrace$ is a basis of ${\bf A}$ and let $f\in \mathbf A$. Let $r(t)=\mathrm R(f,\lbrace f_1,\ldots,f_s\rbrace)$. Then $r(t)\in {\bf A} $. If $r\not=0$, then  $\mathrm{d}(r)\in\langle \mathrm{d}(f_1),\ldots,\mathrm{d}(f_s) \rangle$, because $\lbrace f_1,\ldots,f_r\rbrace$ is a basis, and this is a contradiction.

Conversely, given $0\not=f\in {\bf A} $, if $\mathrm{d}(f)\notin\langle \mathrm{d}(f_1),\ldots,\mathrm{d}(f_s) \rangle$, then $\mathrm R(f,\lbrace f_1,\ldots,f_s\rbrace)\not=0$, which is a contradiction.
\end{proof}

\begin{nota} Suppose that $\lbrace f_1,\ldots,f_s\rbrace$ is a basis of ${\bf A}$. For all $f\in{\bf K}[t]$, $ \mathrm R(f,\lbrace f_1,\ldots,f_s\rbrace)$ is unique. Write $f=g_1+r_1=g_2+r_2$, and suppose that $g_i,r_i$, $i\in\{1,2\}$ satisfy conditions (1), (2) and (3) of Proposition \ref{remainder}. We have $g_1-g_2=r_2-r_1\in {\bf A}$. Hence $\mathrm{d}(r_2-r_1)\in \mathrm{d}(\mathbf A)$, because $\lbrace f_1,\ldots,f_s\rbrace$ is a basis of ${\bf A}$. If $r_1\not= r_2$, then $\mathrm{d}(r_2-r_1)\in \mathrm{supp}(r_1)\cup \mathrm{supp}(r_2)$. Thus by Proposition \ref{remainder}, $\mathrm{d}(r_2-r_1)\in {\mathbb N}\setminus \langle \mathrm{d}(f_1),\ldots,\mathrm{d}(f_s) \rangle= \mathbb N\setminus \mathrm{d}(\mathbf A)$, which is a contradiction.
\end{nota}

Let the notations be as above and let
$$
\phi:\KK[X_1,\ldots,X_s]\longrightarrow\KK[t],\ \phi(X_i)=\mathrm M(f_i), \hbox{ for all } i\in \{1,\ldots,  s\}.$$ 
Let $\lbrace F_1,\ldots,F_r\rbrace$ be a generating system of the kernel of $\phi$. We can choose $F_i$ to be a binomial for all $i\in\{1,\ldots,r\}$. If $F_i=X_1^{\alpha^i_1}\cdots  X_s^{\alpha^i_s}-X_1^{\beta^i_1}\cdots  X_s^{\beta^i_s}$, we set $S_i= f_1^{\alpha^i_1}\cdots  f_s^{\alpha^i_s}-f_1^{\beta^i_1}\cdots  f_s^{\beta^i_s}$. Observe that  if $d=\sum_{k=1}^s\alpha^i_k\mathrm{d}(f_k)=\sum_{k=1}^s\beta^i_k\mathrm{d}(f_k)$, then $\mathrm{d}(S_i)<d$. 

\begin{p}\label{carac-global}
The set $\lbrace f_1,\ldots,f_s\rbrace$ is a basis of ${\bf A}$ if and only if $\mathrm R(S_i,\lbrace f_1,\ldots,f_s\rbrace)=0$ for all $i\in\{1,\ldots,r\}$.
\end{p}
\begin{proof}  Suppose that $\lbrace f_1,\ldots,f_s\rbrace$ is a basis of ${\bf A}$. Since $S_i\in {\bf A}$ for all $i\in\{1,\ldots,r\}$, we trivially obtain $\mathrm R(S_i,\lbrace f_1,\ldots,f_s\rbrace)=0$.

For the sufficiency, assume that there is $f\in {\bf A}$ such that $\mathrm{d}(f)\not\in \langle \mathrm{d}(f_1),\ldots, \mathrm{d}(f_s)\rangle$, and write 
$$
f=\sum\nolimits_{\underline{\theta}}c_{\underline{\theta}}f_1^{\theta_1}\cdots  f_s^{\theta_s}.
$$
For all $\underline{\theta}$, if $c_{\underline{\theta}}\not=0$, we set $p_{\underline{\theta}}=\sum_{i=1}^s\theta_i\mathrm{d}(f_i)=\mathrm{d}(f_1^{\theta_1}\cdots  f_s^{\theta_s})$. Take $p={\rm max}\lbrace p_{\underline{\theta}} \mid c_{\underline{\theta}}\not=0\rbrace$ and let $\lbrace \underline{\theta}^1,\ldots,\underline{\theta}^l\rbrace$ be the set of elements such that $p=\mathrm{d}(f_1^{\theta^i_1}\cdots  f_s^{\theta^i_s})$ for all $i\in\{1,\ldots, l\}$. If $\sum_{i=1}^lc_{\underline{\theta}^i}\mathrm M(f_1^{\theta^i_1}\cdots  f_s^{\theta^i_s})\not=0$, then $p=\mathrm{d}(f)\in \langle \mathrm{d}(f_1),\ldots,\mathrm{d}(f_s) \rangle$, which by assumption is impossible. Hence $\sum_{i=1}^lc_{\underline{\theta}^i}\mathrm M(f_1^{\theta^i_1}\cdots  f_s^{\theta^i_s})=0$, which implies that  $\sum_{i=1}^lc_{\underline{\theta}^i}X_1^{\theta^i_1}\cdots  X_s^{\theta^i_s}\in \ker(\phi)$. Thus 
$$
\sum_{i=1}^lc_{\underline{\theta}^i}X_1^{\theta^i_1}\cdots  X_s^{\theta^i_s}=\sum_{k=1}^r\lambda_kF_k
$$
with $\lambda_k\in\KK[X_1,\ldots,X_s]$ for all $k\in\{1,\ldots,r\}$. This implies that 
$$
\sum_{i=1}^lc_{\underline{\theta}^i}f_1^{\theta^i_1}\cdots  f_s^{\theta^i_s}=\sum_{k=1}^r\lambda_k(f_1,\ldots,f_s)S_k.
$$
By hypothesis, $\mathrm R(S_k,\{f_1,\ldots, f_s\})=0$. So there exists an expression $S_k=\sum_{\underline{\beta}^k}c_{\underline{\beta}^k}f_1^{\beta_1^k}\cdots  f_s^{\beta_s^k}$
with $\mathrm{d}(f_1^{\beta_1^k}\cdots  f_s^{\beta_s^k})\leq \mathrm{d}(S_k)$ for all $\underline{\beta}^k$ such that $c_{\underline{\beta}^k}\not= 0$.  Finally we can write $f=\sum_{\underline{\theta}'}c_{\underline{\theta}'}f_1^{\theta'_1}\cdots  f_s^{\theta'_s}$ with $\max \left\lbrace \mathrm{d}(f_1^{\theta'_1}\cdots  f_s^{\theta'_s})\mid c_{\underline{\theta}'}\not=0\right\rbrace<p$. 

We now restart with the new expression of $f$. This process will stop, yielding a contradiction.
\end{proof}

\begin{alg}
Let the notations be as above.

\begin{enumerate}
\item If $\mathrm R(S_k(f_1,\ldots,f_s),\lbrace f_1,\ldots,f_s\rbrace)=0$ for all $k\in\{1,\ldots,r\}$, then $\lbrace f_1,\ldots,f_s\rbrace$ is a basis of ${\bf A}$.

\item If $r(t)=\mathrm R(S_k(f_1,\ldots,f_s),\lbrace f_1,\ldots,f_s\rbrace)\not= 0$ for some $k\in\{1,\ldots, r\}$,  then we set $f_{s+1}=r(t)$, and we restart with $\lbrace f_1,\ldots,f_{s+1}\rbrace$. \\
Note that in this case, $\langle \mathrm{d}(f_1),\ldots,\mathrm{d}(f_s)\rangle \subsetneq \langle\mathrm{d}(f_1),\ldots,\mathrm{d}(f_s),\mathrm{d}(f_{s+1}) \rangle$. 
\end{enumerate}
This process will stop, giving a basis of ${\bf A}$.
\end{alg}

Suppose that $\lbrace f_1,\ldots,f_s\rbrace$ is a basis of ${\bf A}$. We say that $\lbrace f_1,\ldots,f_s\rbrace$ is  a \emph{minimal basis} of ${\bf A}$ if $\{\mathrm{d}(f_1),\ldots,\mathrm{d}(f_s)\}$ minimally generates the semigroup $\mathrm{d}({\bf A})$. We say that $\lbrace f_1,\ldots,f_s\rbrace$ is a \emph{reduced basis} of ${\bf A}$ if $\mathrm{supp}(f_i-\mathrm M(f_i))\in \NN\setminus \mathrm{d}({\bf A})$ and $f_i$ is monic for all $i\in\{1,\ldots,s\}$.

Let $i\in\{1,\ldots, s\}$. If $\mathrm{d}(f_i)$ is in $\langle \mathrm{d}(f_1),\ldots,\mathrm{d}(f_{i-1}),\mathrm{d}(f_{i+1}),\ldots,\mathrm{d}(f_s) \rangle$, then the set obtained by removing $f_i$, $\lbrace f_1,\ldots,f_{i-1},f_{i+1},\ldots,f_s\rbrace$, is also a basis of ${\bf A}$.
Furthermore, by applying the division process of Proposition \ref{remainder} to $f_i-\mathrm M(f_i)$, we can always construct a reduced basis of ${\bf A}$. 

\begin{corolario} 
Up to constants, the algebra ${\bf A}$ has a unique minimal reduced  basis.
\end{corolario}

\begin{proof} Let $\lbrace f_1,\ldots,f_s\rbrace$ and $\lbrace g_1,\ldots,g_{s'}\rbrace$ be two minimal reduced bases of ${\mathbf A}$. Clearly $s=s'$, and equals the embedding dimension of $\mathrm{d}({\mathbf A})$. Let $i=1$. There exists $j_1$ such that $\mathrm{d}(f_1)=\mathrm{d}(g_{j_1})$, because minimal generating systems of a numerical semigroup are unique. 

Observe that $\supp(f_1-g_{j_1})\subseteq \supp(f_1-\mathrm M(g_{j_1}))=\supp(f_1-\mathrm M(f_1))\subseteq \mathbb N\setminus \mathrm{d}({\mathbf A})$. 
Thus, if $f_1-g_{j_1}\notin \KK\setminus\{0\}$, then $\mathrm{d}(f_1-g_{j_1})\notin \mathrm{d}({\mathbf A})$, which is a contradiction because $f_1-g_{j_1}\in {\mathbf A} $. The same argument shows that for all $i\geq 2$, there exists $j_i$ such that $f_i-g_{j_i}\in\KK$.
\end{proof}

\begin{corolario} Let $\lbrace f_1,\ldots,f_s\rbrace$ be a reduced basis of $A$. For all $i\in\{1,\ldots,s\}$,  $\mathrm{supp}(f_i- \mathrm M(f_i))\subseteq  \mathrm G(d({\mathbf A})$. 
\end{corolario}

\begin{ex}\label{ex-1}
We compute $\mathrm{d}({\bf A})$ for ${\bf A}=\KK[t^6+t,t^4]$; $f_1=t^6+t$ and $f_2=t^4$. We start by computing the kernel of  $\phi:\KK[X_1,X_2]\to \KK[t]$, with $\phi(X_1)=t^6$ and $\phi(X_2)=t^4$. This kernel is generated by $F_1=X_2^3-X_1^2$. Hence $S_1=2t^7+t^2$. Since $7\notin \langle 4,6\rangle$, then we add $f_3=2t^7+t^2$ to our generating set. 

In the next step $\phi:\KK[X_1,X_2,X_3]\to \KK[t]$, with $\phi(X_1)=t^6$, $\phi(X_2)=t^4$ and $\phi(X_3)=2t^7$; $\ker\phi=(X_2^3-X_1^2, X_3^2-4X_1X_2^2)$, whence $S_1=f_3$ and $S_2=f_3^2-4f_1f_2^2=t^4=f_2$. It turns out that $\mathrm R(S_1,\{f_1,f_2,f_3\})=0=\mathrm R(S_2,\{f_1,f_2,f_3\})$, and consequently $\{f_1,f_2,f_3\}$ is a (reduced minimal) basis of ${\bf A}$. Also $\mathrm{d}({\bf A})=\langle 4,6,7\rangle$.

These computations can be performed with the \texttt{numericalsgps} \texttt{GAP} package.
\begin{verbatim}
gap> SemigroupOfValuesOfCurve_Global([t^6+t,t^4],"basis");
[ t^4, t^6+t, t^7+1/2*t^2 ]
\end{verbatim}
Or if we just want to calculate $\mathrm{d}(A)$:
\begin{verbatim}
gap> s:=SemigroupOfValuesOfCurve_Global([t^6+t,t^4]);;              
gap> MinimalGenerators(s);
[ 4, 6, 7 ]
\end{verbatim}
\end{ex}

\section{Modules over ${\bf K}$-algebras}

Let the notations be as in Section \ref{bases of algebras}. In particular $\lbrace f_1,\ldots,f_s\rbrace$ is a set of polynomials of $\KK[t]$ and ${\bf A}=\KK[f_1,\ldots,f_s]$. Let $\lbrace F_1,\ldots,F_r\rbrace$ be a set of nonzero elements of $\KK[t]$, and let ${\bf M}=\sum_{i=1}^rF_i{\bf A}$ be the ${\bf A}$-module generated by $F_1,\ldots,F_r$. We set $\mathrm{d}({\bf M})=\lbrace \mathrm{d}(F),F\in{\bf M}\setminus 0\rbrace$. 

If $F\in{\bf M}$ and $g\in {\bf A}$ then $gF\in{\bf M}$, hence $\mathrm{d}({\bf M})$ is a relative ideal of $\mathrm{d}({\bf A})$. 

\begin{definicion} We say that $\lbrace F_1,\ldots,F_r\rbrace$ is a basis of $\mathbf M$ if and only if $\mathrm{d}(M)=\bigcup_{i=1}^r (\mathrm{d}(F_i)+\mathrm{d}(\mathbf A))$. Equivalently, $\lbrace F_1,\ldots,F_r\rbrace$ is a basis of $\bf M$ if and only if $\lbrace \mathrm{d}(F_1),\ldots,\mathrm{d}(F_r)\rbrace$ is a basis of the ideal $\mathrm{d}(\mathbf M)$ of $\mathrm{d}(\mathbf A)$.
\end{definicion}

\begin{teorema} \label{division-ideal}
Let $\lbrace f_1,\ldots,f_s,F_1,\ldots, F_r\rbrace$ be a set of nonzero polynomials of $\KK[t]$. Let ${\bf A}=\KK[f_1,\ldots,f_s]$ and ${\bf M}$ be the ${\bf A}$-module generated by $\{F_1,\ldots,F_r\}$. Given $F\in{\KK}[t]$, $F\not=0$, there exist $g_1,\ldots,g_r\in {\bf A}$ and $R\in \mathbf K[t]$ such that the following conditions hold.

\begin{enumerate}[(1)]

\item $F=\sum_{i=1}^rg_iF_i+R$.

\item For all $i\in\lbrace 1,\ldots,r\rbrace$, if $g_i\not=0$, then $\mathrm{d}(g_i)+\mathrm{d}(F_i)\leq \mathrm{d}(F)$.

\item If $R\not=0$, then $\mathrm{d}(R)\leq\mathrm{d}(F)$ and $\mathrm{d}(R)\in {\mathbb N}\setminus\bigcup_{i=1}^r(\mathrm{d}(F_i)+\mathrm{d}(\mathbf A))$.
\end{enumerate}
\end{teorema}

\begin{proof} The assertion is clear if $F\in\KK$. Suppose that $F\notin\KK$ and let $F(t)=\sum_{i=0}^pc_it^i$ with $p=\mathrm{d}(f)>0$.

In order to simplify notation, set $S=\mathrm{d}(\mathbf A)$ and $I=\bigcup_{i=1}^r (\mathrm{d}(F_i)+S)$.

\begin{enumerate}[(i)]

\item If $p\notin I$, then we set $g^1=\dots=g^r=0$,  $R^1=c_pt^p$ and $F^1=F-R^1$.

\item If $p\in I$, then  $c_pt^p=c_{\theta_i}t^{s_i}\mathrm M(F_i)$ for some $s_i\in S$ and some $i\in\lbrace 1,\ldots,r\rbrace$. Let $g\in{\bf A}$ such that $\mathrm M(g)=c_{\theta_i}t^{s_i}$. We set $g_i^1=g$, $g_j^1=0$ for all $j\not=i$, $R^1=0$ and $F^1=F-gF_i$.
\end{enumerate}

In this way, $F=F^1+\sum_{i=1}^rg_i^1F_i+R^1$, and the following conditions hold:

\begin{enumerate}

\item  $g_i^1\in\bf A$ for all $i\in\lbrace 1,\ldots,r\rbrace$.

\item If $R^1\not=0$, then $\mathrm{supp}(R^1)\subseteq \NN\setminus I$.

\item If $F^1\notin \KK$, then $\mathrm{d}(F^1)<\mathrm{d}(F)=p$. 
\end{enumerate}
Then we restart with $F^1$. Clearly there is $k\geq 1$ such that $F^{k}\in \KK$. We set
$g_i=g_i^1+\dots+g_i^k$ for all $i\in\lbrace 1,\ldots,r\rbrace$, and $R=R^1+\cdots+R^k+F^k$.
\end{proof}
We denote the polynomial $R$ of Theorem \ref{division-ideal} by $\mathrm R_{\bf A}(F,\lbrace F_1,\ldots,F_r\rbrace)$.

The following \texttt{GAP} code can compute $\mathrm R_{\mathbf A}(f,\{F_1,\ldots,F_r\})$. Here \texttt{A} contains a basis of the algebra ${\bf A}$, and \texttt{M}  is $\{F_1,\ldots, F_r\}$.

\begin{lstlisting}[basicstyle=\ttfamily, tabsize=2, label=reduce-code, caption={\texttt{reduce} function}, language=gap]
reduce:=function(A,M,f)
	local gens,geni,cand,d, fact, c, r, s,a;
	gens:=List(A, DegreeOfLaurentPolynomial);
	s:=NumericalSemigroup(gens);
	geni:=List(M,DegreeOfLaurentPolynomial);
	if IsZero(f) then
		return f;
	fi;
	d:=DegreeOfLaurentPolynomial(f);
	c:=First([1..Length(geni)], i->d-geni[i] in s);
	r:=f;
	while c<>fail do
		fact:=FactorizationsIntegerWRTList(d-geni[c],gens);
		a:=M[c]*Product(List([1..Length(gens)],i->A[i]^fact[1][i]));
		r:=r-LeadingCoefficient(r)*a/LeadingCoefficient(a);
		if IsZero(r) then
			return r;
		fi;
		d:=DegreeOfLaurentPolynomial(r);
		c:=First([1..Length(geni)], i->d-geni[i] in s);
	od;
	return r/LeadingCoefficient(r);
end;
\end{lstlisting}

\begin{proposicion}\label{caracterization} Let the notations be as in Theorem \ref{division-ideal}. The following conditions are equivalent:

\begin{enumerate}
\item $\lbrace F_1,\ldots,F_r\rbrace$ is a basis of  ${\bf M}$.

\item For all $F\in{\bf M}$, $\mathrm R_{\bf A}(F,\lbrace F_1,\ldots,F_r\rbrace)=0$.
\end{enumerate}
\end{proposicion}

\begin{proof} 
Suppose that $\lbrace F_1,\ldots,F_r\rbrace$ is a basis of ${\bf M}$ and let $F\in {\bf M}$. If $R=\mathrm R_{\bf A}(F,\lbrace F_1,\ldots,F_r\rbrace)\not=0$, then $\mathrm{d}(R)\in {\mathbb N}\setminus \bigcup_{i=1}^r(\mathrm{d}(F_i)+\mathrm{d}(\mathbf A))=\mathbb N\setminus \mathrm{d}(\mathbf M)$. But $R\in{\bf M}$. This is a contradiction.

Conversely suppose that $\mathrm R_{\bf A}(F,\lbrace F_1,\ldots,F_r\rbrace)=0$ for all $F\in{\bf M}$. Take $F\in{\bf M}$. If $\mathrm{d}(F) \in {\mathbb N}\setminus \bigcup_{i=1}^r(\mathrm{d}(F_i)+\mathrm{d}(\mathbf A))$, then by construction, $\mathrm R_{\bf A}(F,\lbrace F_1,\ldots,F_r\rbrace)\not=0$, which is a contradiction.
\end{proof}



Let $F_1,\ldots,F_r\in\KK[t]$ and assume, without loss of generality, that $F_1,\ldots,F_r$ are monic. Assume also that $\{f_1,\ldots,f_s\}$ is a reduced basis for $\mathbf A$. Let $a_i$ be such that $\mathrm M(F_i)=t^{a_i}$ for all $i\in\lbrace 1,\ldots,r\rbrace$. Let $(s_i,s_j)\in \mathrm R(a_i,a_j)$. Then $s_i,s_j\in \mathrm{d}(\mathbf A)$. Then $s_i=\sum_{l=1}^s e_{i_l}\mathrm{d}(f_l)$ and  $s_j=\sum_{l=1}^s e_{j_l}\mathrm{d}(f_l)$, for some $e_{i_l}, e_{j_l}\in \mathbb N$. Let $g_i=\prod_{l=1}^s f_l^{e_{i_l}}$, $g_j=\prod_{l=1}^s f_l^{c_{j_l}} \in {\bf A}$. Note that these polynomials may not be unique, there are as many as factorizations of $s_i$ and $s_j$, but this amount is finite. Then $\mathrm{d}(g_i)=s_i$ and $\mathrm{d}(g_j)=s_j$, and also $\mathrm M(g_i)=t^{s_i}$ and $\mathrm M(g_j)=t^{s_j}$ (recall that $f_l$ is monic for all $l$).  We have $t^{s_i}\mathrm M(F_i)-t^{s_j} \mathrm M(F_j)=0$, whence $t^{s_i}{\bf e}_i-t^{s_j}{\bf e}_j\in \ker(\phi)$ with $\phi: {\bf A}^r\to {\bf M}$, $\phi(p_1,\ldots,p_r)=\sum_{i=1}^rp_i\mathrm M(F_i)$. If 
$$
F=g_iF_i-g_jF_j,
$$
then $\mathrm{d}(F)<a_i+s_i=a_j+s_j$. We call $F$ an $S$-polynomial of $(F_1,\ldots,F_r)$. Every element of $\ker(\phi)$ gives rise to an $S$-polynomial. Let $\mathrm{SP}(F_1,\ldots,F_r)$ be the set of $S$-polynomials of $(F_1,\ldots,F_r)$ constructed this way. The set $\mathrm{SP}(F_1,\ldots,F_r)$ has finitely many elements, though for our purposes it will be enough to choose a finite subset of $\mathrm{SP}(F_1,\ldots,F_r)$. 

Let $n\in \mathrm{d}(A)$. The set $\mathsf Z(n)=\{ (n_1,\ldots,n_s) \in \mathbb N^s\mid n=\sum_{i=1}^s n_i\mathrm{d}(g_i)\}$ has finitely many elements (usually known as the set of factorizations of $n$). Let $\preceq_{\mathsf{lex}}$ denote the lexicographical ordering in $\mathbb N^s$. We will consider $\mathrm{MinSP}(F_1,\ldots, F_r)$ the set of all elements $g_iF_i-g_jF_j\in \mathrm{SP}(F_1,\ldots,F_r)$ such that, with the above notation, $g_i=\prod_{l=1}^s f_l^{e_{i_l}}$ and $g_j=\prod_{l=1}^s f_l^{c_{j_l}}$ with $(e_{i_1},\ldots,e_{i_s})=\min_{\preceq_{\mathsf{lex}}}(\mathsf Z(\mathrm{d}(g_i))$ and $(e_{j_1},\ldots,e_{j_s})=\min_{\preceq_{\mathsf{lex}}}(\mathsf Z(\mathrm{d}(g_j))$.

In Theorem \ref{S-caracterization} we give a characterization for a set $\lbrace F_1,\ldots,F_r\rbrace$ of ${\bf M}$ to be a basis of ${\bf M}$ in terms of $\mathrm{MinSP}(F_1,\ldots,F_r)$. 





\begin{teorema} \label{S-caracterization} Let $\lbrace f_1,\ldots,f_s,F_1,\ldots, F_r\rbrace$ be a set of nonzero polynomials of $\KK[t]$. Let ${\bf A}=\KK[f_1,\ldots,f_s]$ and ${\bf M}$ be the ${\bf A}$-module generated by $\{F_1,\ldots,F_r\}$. The following conditions are equivalent:

\begin{enumerate}

 \item $\lbrace F_1,\ldots,F_r\rbrace$ is a basis of ${\bf M}$,

\item For all $F\in   \mathrm{MinSP}(F_1,\ldots,F_r)$, $\mathrm R_{\bf A}(F,\lbrace F_1,\ldots,F_r\rbrace)=0$.

\end{enumerate}
 
\end{teorema} 

\begin{proof} 
In order to simplify notation, set $a_i=\mathrm{d}(F_i)$ for all $i\in \{1,\ldots,r\}$, and $S=\mathrm{d}(\mathbf A)$.

\emph{(1)} implies  \emph{(2)} follows from Proposition \ref{caracterization}. 

For the other implication, we are going to show that for each $R\in \mathbf M$, $\mathrm{d}(R)\in  \bigcup_{i=1}^r(a_i+S)$. 


Take $R\in \mathbf M$. If $R=0$, we are done. Otherwise, we can find an expression of the form $R=g_1F_1+\dots+g_rF_r$ with $g_1,\ldots,g_r\in{\bf A}$. Assume that  $\mathrm{d}(R)\in {\mathbb N}\setminus \bigcup_{i=1}^r(a_i+S)$. Set 
\[
p=\max\nolimits_{i,\ g_i\not=0}(a_i+\mathrm \alpha_i). 
\]
Where $\alpha_i=\mathrm{d}(g_i)$, $i\in\{1,\ldots,r\}$. Then $p>\mathrm{d}(R)$. We shall prove that there exists another expression of $R$, say $R=g'_1F_1+\dots+g'_rF_r$ with $p>p'=\max_{i,\  g'_i\not=0}(a_i+ \mathrm{d}(g'_i))$. And this eventually leads to a contradiction, since the interval $\{\mathrm{d}(R)+1,\ldots, p\}$ has finitely many elements.

Suppose, without loss of generality, that $p=\mathrm \alpha_i+a_i$, $i\in\{1,\ldots,l\}$ and $p>\mathrm \alpha_i+a_i$, $\in\{l+1,\ldots,r\}$. Clearly $l\geq 2$. We prove by induction on $l$ that we can rewrite $R$ as $R=g'_1F_1+\dots+g'_rF_r$ with $p>p'=\max_{i, g'_i\not=0}(\mathrm{d}(g'_i)+a_i)$.

\begin{enumerate}[(i)]
\item We first suppose that $l=2$ and let $\mathrm M(g_1)=c_{g_1}t^{\alpha_1}$, $\mathrm M(g_2)=c_{g_2}t^{\alpha_2}$. It follows from the hypothesis that $c_{g_2}=-c_{g_1}$, and  also that $\alpha_1=s+s_1, \alpha_2=s+s_2$ with $(s_1,s_2)\in \mathrm R(a_1,a_2)$. Hence we have 
\[
c_{g_1}t^{\alpha_1}t^{a_1}+c_{g_2}t^{\alpha_2}t^{a_2}=c_{g_1}t^s(t^{s_1}t^{a_1}-t^{s_2}t^{a_2}).
\] 
Let $\tilde{g}_1,\tilde{g}_2\in{\bf A}$ such that $\mathrm M(\tilde{g}_1)=t^{s_1}$, $\mathrm M(\tilde{g}_2)=t^{s_2}$, and $\tilde{g}_1F_1-\tilde{g}_2F_2$ is a minimal $S$-polynomial. We have $\mathrm{d}(\tilde{g}_1F_1-\tilde{g}_2F_2)<s_1+a_1= \alpha_1-s+a_1=p-s$. By hypothesis, $\mathrm R_{\mathbf A}(\tilde{g}_1F_1-\tilde{g}_2F_2, \{F_1,\ldots, F_r\})=0$, and thus  \[\tilde{g}_1F_1-\tilde{g}_2F_2=\bar{g}_1F_1+\bar{g}_2F_2+\dots+\bar{g_r}F_r,\] with $\mathrm{d}(\bar{g_i}F_i)\leq \mathrm{d}(\tilde{g}_1F_1-\tilde{g}_2F_2)<p-s$ for all $i\in\lbrace 1,\ldots,r\rbrace$. We can then rewrite $R$ as  
\begin{align*}
R&=(g_1-c_{g_1}t^s\tilde{g}_1)F_1+(g_2-c_{g_2}t^s\tilde{g}_2)F_2+c_{g_1}t^s\bar{g}_1F_1+c_{g_2}t^s\bar{g}_2F_2+\sum\nolimits_{i\geq 3}g_iF_i\\ &=\sum\nolimits_{i=1}^r{g}_i' F_i,
\end{align*}
with $\mathrm{d}({g}'_iF_i)<p$ for all $i\in\lbrace 1,\ldots,r\rbrace$. 

\item Now let $l>2$ and let $\mathrm M(g_i)=c_{g_i}t^{s_i}$ for all $i\in\lbrace 1,\ldots,r\rbrace$. We have  $R=\sum_{i=1}^rg_iF_i=g_1F_1-\frac{c_{g_1}}{c_{g_2}}g_2F_2+(\frac{c_{g_1}}{c_{g_2}}+1)g_2F_2+\sum_{i=3}^rg_iF_i$. It follows from (i) that $g_1F_1-\frac{c_{g_1}}{c_{g_2}}g_2 F_2= \bar{g}_1F_1+\dots+\bar{g}_rF_r$ with $\max_{i, \bar{g}_i\not=0}\mathrm{d}(\bar{g}_iF_i)< p$. Hence $R=\tilde{g}_1F_1+\ldots+\tilde{g}_rF_r$ with 
\begin{itemize}
\item $\tilde{g}_1=\bar{g}_1$,   
\item $\tilde{g}_2=\bar{g}_2+(\frac{c_{g_1}}{c_{g_2}}+1)g_2$,
\item $\tilde g_i=\bar g_i+g_i$ for $i\in \{3,\ldots, r\}$.
\end{itemize}
In particular, the set $\lbrace i \mid \mathrm{d}(\tilde{g}_iF_i)=p\rbrace$ has at most $l-1$ elements, and it follows from the induction hypothesis that 
$R=g'_1F_1+\ldots+g'_rF_r$ with $p>p'={\mathrm max}_{i, g'_i\not=0}(\mathrm{d}(g'_i)+\mathrm{d}(F_i))$. 
\qedhere
\end{enumerate}
\end{proof}

\begin{alg} Let ${\bf M}=\sum_{i=1}^rF_i{\bf A}$. 

\begin{enumerate}

\item If for all $F\in \mathrm{MinSP}(F_1,\ldots,F_r)$, $\mathrm R_{\bf A}(F,\lbrace F_1,\ldots,F_r\rbrace)=0$ then, by Theorem \ref{S-caracterization}, $\lbrace F_1,\ldots,F_r\rbrace$ is a basis of ${\bf M}$. Return $\{F_1,\ldots, F_r\}$.

\item If $\mathrm R_{\bf A}(F,\lbrace F_1,\ldots,F_r\rbrace)\neq 0$ for some $F\in SP(F_1,\ldots,F_r)$, then we set $F_{r+1}=\mathrm R_{\bf A}(F,\lbrace F_1,\ldots,F_r\rbrace)$ and we restart with $\lbrace F_1,\ldots,F_{r+1}\rbrace$.
\end{enumerate}
Since ${\mathbb N}\setminus \bigcup_{i=1}^r(\mathrm{d}(F_i)+S)$ has finitely many elements, then the procedure stops after a finite number of steps, returning a basis of ${\bf M}$.
\end{alg}

\begin{lstlisting}[basicstyle=\ttfamily, tabsize=2, caption={generatorsModule}]
generatorsModule:=function(A,M,t)
	local S, gens, gM, a, b, da, db, i, j, rs, rd, rel, fcta, fctb, C, 
		pair, reduction, n;
	
	gens:=List(A, DegreeOfLaurentPolynomial);
	n:=Length(A);
	S:=NumericalSemigroup(gens);
	
	gM:=ShallowCopy(M);
	C:=[];
	for i in [1..Length(gM)] do
	  for j in [i+1..Length(gM)] do
	    Add(C,[gM[i],gM[j]]);
	  od;
	od;
	while C<>[] do
		pair:=Remove(C,1);
		a:=pair[1];
		b:=pair[2];
		da:=DegreeOfLaurentPolynomial(a);
		db:=DegreeOfLaurentPolynomial(b);
		rs:=R(da,db,S);
		reduction:=true;
		for rel in rs do
			fcta:=FactorizationsIntegerWRTList(rel[1],gens)[1];
			fctb:=FactorizationsIntegerWRTList(rel[2],gens)[1];
			rd:=reduce(A,gM,
				a*Product(List([1..n], i->A[i]^fcta[i]))-
				b*Product(List([1..n], i->A[i]^fctb[i])));
			if not(IsZero(rd)) then
				C:=Union(C,List(gM, x->[x,rd]));
				Add(gM,rd);
			fi;
		od;
	od;
	
	reduction:=false;
	while not reduction do
		reduction:=true;
		a:=First(gM, x->x<>reduce(A,Difference(gM,[x]),x));
		if a<>fail then
			rd:=reduce(A,Difference(gM,[a]),a);
			if IsZero(rd) then
				gM:=Difference(gM,[a]);
			else
				gM:=Union(Difference(gM,[a]),[rd]);
			fi;
			reduction:=false;
		fi;
	od;
	return gM;
end;
\end{lstlisting}

\begin{ex} \label{ex-2} Let ${\bf A}=\KK[t^6+t,t^4]$ be as in Example \ref{ex-1}, and recall that  $\lbrace f_1=t^6+t, f_2=t^4, f_3=t^7+\frac{1}2 t^2\rbrace$ is a basis of $\mathrm{d}({\bf A})$. Let ${\bf M}=F_1{\bf A}+F_2{\bf A}$ with $F_1=t^3$ and $F_2=t^4$. We have $(3+\mathrm{d}(\mathbf A))\cap (4+\mathrm{d}(\mathbf A))=\{10,11\}+\mathrm{d}(\mathbf A)$. Thus $\mathrm R(3,4)=\{(7,6),(8,7)\}$. 

For $(7,6)$, $7=\mathrm{d}(f_3)$ and $6=\mathrm{d}(f_1)$ (and these are the only factorizations of $7$ and $6$ in terms of the generators of $\mathrm{d}(\mathbf A)$). We have the S-polynomial \[f_3F_1-f_1F_2=\left(t^7+\frac{1}2t^2\right)t^3-(t^6+t)t^4=-\frac{1}2t^5.\]
We take $F_3=t^5$, and as $5\not\in \{3,4\}+\mathrm{d}(\mathbf A)$, we add it to our system of generators, obtaining $\{F_1,F_2,F_3\}$.

Now for $(8,7)$ we have the S-polynomial 
\[
f_2^2F_1-f_3F_2 =t^8t^3-\left(t^7+\frac{1}2 t^2\right)t^4=-\frac{1}2 t^6.
\]
Set $F_4=t^6$. As $6\not\in \{3,4,5\}+\mathrm{d}(\mathbf A)$, we add it to our generating set: $\{F_1,F_2,F_3,F_4\}$. One can show that any other S-polynomial with respect to this new generating system reduces to zero, and thus $\{F_1,F_2,F_3,F_4\}$ is a basis for $\mathrm{d}(\mathbf M)$.

\begin{verbatim}
gap> A:=SemigroupOfValuesOfCurve_Global([t^6+t,t^4],"basis");
[ t^4, t^6+t, t^7+1/2*t^2 ]
gap> M:=[t^3,t^4];;                                                             
gap> generatorsModule(A,M,t);
[ t^3, t^4, t^5, t^6 ]
gap> SetInfoLevel(InfoNumSgps,2);
gap> generatorsModule(A,M,t);
#I  new generator t^5 of degreee 5
#I  new generator t^6 of degreee 6
#I  Reducing...
[ t^3, t^4, t^5, t^6 ]
\end{verbatim}
\end{ex}

\section{Module of K\"ahler differentials}

Let $\lbrace f_1,\ldots,f_r\rbrace$ be a set of polynomials of $\KK[t]$ and ${\bf A}=\KK[f_1,\ldots,f_r]$. We shall assume that ${\mathbb N}\setminus \mathrm{d}({\bf A})$ is a finite set, in particular $\mathrm{d}({\bf A})$ is a numerical semigroup. We shall denote it by $S$. Let $F_i=f_i'(t)$ for all $i\in\lbrace 1,\ldots,r\rbrace$, and let $\mathbf M=F_1 \mathbf A+\dots+F_r \mathbf A$. We know that the set $I=\mathrm{d}({\bf M})=\lbrace \mathrm{d}(F) \mid F\in {\bf M}\rbrace$ is a relative ideal of $S$. 

Given $g\in{\bf A}$, we have $g'(t)\in {\bf M}$. In particular, if $s\in S$, then $s-1\in I$. We say that $s-1$ is an \emph{exact degree}. We call the other elements of $I$ \emph{non exact degrees} of $\mathbf M$. We denote by $\mathrm{NE}({\bf M})$ the set of non exact degrees, that is 
\[
\mathrm{NE}(\mathbf M)=\{ i\in I \mid i+1\not\in S\}.
\] 
Let $\mathrm{ne}(\mathbf M)$ be the cardinality of $\mathrm{NE}({\bf M})$. It follows that $\mathrm{ne}(\mathbf M)\leq \mathrm g(S)$, the genus of $S$.

\begin{ex}
Let $x(t)$ and $y(t)$ be polynomials of degree $3$ and $4$ respectively. As $\gcd(3,4)=1$,  $\{x(t), y(t)\}$ is a basis for ${\bf A}=\mathbb K[x(t),y(t)]$ and $S=\mathrm{d}({\bf A})=\langle 3,4\rangle$. Set $\mathbf M=x'(t){\bf A}+y'(t){\bf A}$. Then $I=\mathrm{d}({\bf M})$ contains the ideal $J=(2,3)+S$. The lattice of ideals of $S$ containing $J$ is the following. 
\begin{center}
	
	\begin{verbatim}
	gap> s:=NumericalSemigroup(3,4);;
	gap> oi:=overIdeals([2,3]+s);
	[ <Ideal of numerical semigroup>, <Ideal of numerical semigroup>, 
	<Ideal of numerical semigroup>, <Ideal of numerical semigroup>, 
	<Ideal of numerical semigroup> ]
	gap> List(oi,MinimalGenerators);
	[ [ 2, 3 ], [ 0, 1, 2 ], [ 0, 2 ], [ 1, 2, 3 ], [ 2, 3, 4 ] ]
	\end{verbatim}	
	
	
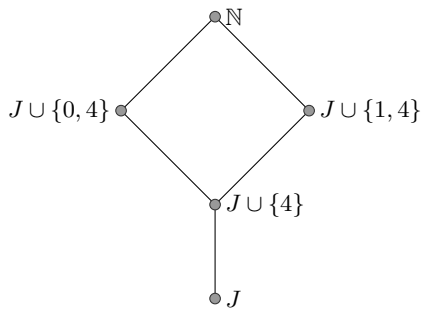
\begin{tikzpicture}[y=.25cm, x=.25cm,font=\footnotesize]

	\draw (5,-3) -- (5,-8);
	
	\draw (5,-3) -- (0,2);
	
	\draw (5,-3) -- (10,2);
	
	\draw (0,2) -- (5,7);
	
	\draw (10,2) -- (5,7);
	
	
	
	\filldraw[fill=black!40,draw=black!80] (5,-3) circle (2pt)    node[right] {$J\cup\{4\}$};

	\filldraw[fill=black!40,draw=black!80] (0,2) circle (2pt)    node[left] {$J\cup\{0,4\}$};
	
	\filldraw[fill=black!40,draw=black!80] (10,2) circle (2pt)    node[anchor=west] {$J\cup\{1,4\}$};
	
	\filldraw[fill=black!40,draw=black!80] (5,7) circle (2pt)    node[anchor=west] {$\mathbb N$};
	
	\filldraw[fill=black!40,draw=black!80] (5,-8) circle (2pt)    node[right] {$J$};
	\end{tikzpicture}
\end{center}

The set of non exact elements for each ideal is:
\begin{verbatim}
gap> List(oi,non exactElements); 
[ [  ], [ 0, 1, 4 ], [ 0, 4 ], [ 1, 4 ], [ 4 ] ]
\end{verbatim}

And all these ideals can be realized as $\mathrm{d}(M)$ for some $x(t)$, $y(t)$.
\begin{itemize}
	\item $J=(2,3)+S=I$ for $(x(t),y(t))= (t^3,t^4)$.
	\item $J\cup\{4\}= (2,3,4)+S=I$ for $(x(t),y(t))= (t^3+t^2,t^4)$.
	\item $J\cup\{0,4\}= (0,2)+S=I$ for $(x(t),y(t))= (t^3,t^4+t)$.
	\item $J\cup\{1,4\}= (1,2,3)+S= \mathrm I$ for $(x(t),y(t))= (t^3,t^4+t^2)$.
	\item $\mathbb N=J\cup\{0,1,4\}= (0,1,2)+S=\mathrm I$ for $(x(t),y(t))= (t^3+t,t^4)$.
\end{itemize}

\begin{verbatim}
gap> A:=[t^3,t^4];;
gap> generatorsKhalerDifferentials(A,t);
[ t^2, t^3 ]
gap> A:=[t^3+t^2,t^4];;
gap> generatorsKhalerDifferentials(A,t);
[ t^2+2/3*t, t^3, t^4 ]
gap> A:=[t^3,t^4+t];;
gap> generatorsKhalerDifferentials(A,t);
[ 1, t^2 ]
gap> A:=[t^3,t^4+t^2];;
gap> generatorsKhalerDifferentials(A,t);
[ t, t^2, t^3 ]
gap> A:=[t^3+t,t^4];;    
gap> generatorsKhalerDifferentials(A,t);
[ 1, t, t^2 ]
\end{verbatim}
\end{ex}

\begin{proposicion} Let the notations be as above. If $S$ is symmetric, then  $\mathrm{ne}(\mathbf M)\leq \frac{{\mathrm{F}}(S)}{2}$.
\end{proposicion}
\begin{proof} 
In fact, the cardinality of $\lbrace s \mid s+1\notin S \rbrace$ is, in this case, $\frac{{\mathrm{F}}(S)}{2}$ (see for instance \cite[Chapter 3]{ns}).
\end{proof}

In the following we shall suppose that $r=2$, and that $\KK$ is an algebraically closed field of characteristic zero. We shall also use the notation $X(t), Y(t)$ for $f_1(t),f_2(t)$ and we recall that $\lambda_{\bf A}({\bf K}[t]/{\bf K}[X(t),Y(t)]<+\infty$. Let $f(X,Y)$ be the monic generator of the kernel of the map $\psi:\KK[X,Y]\to \KK[t], \psi(X)=X(t),\psi(Y)=Y(t)$. Then $f$ has one place at infinity (see \cite{ab}). 

We shall denote $S={\mathrm{d}}({\bf A})$ by $\Gamma(f)$. Given a nonzero polynomial $g(X,Y)\in\KK[X,Y]$, the element ${\mathrm{d}eg}_tg(X(t),Y(t))$ of $\Gamma(f)$ coincides with the rank over $\KK$ of the $\KK$-vector space $\frac{\KK[X,Y]}{(f,g)}$ (see for instance \cite[Chapter 4]{ag}). 

Let $f_X,f_Y$ denote the partial derivatives of $f$ and let $(f-\lambda)_{\lambda\in\KK}$ be the family of translates of $f$. Let $\lambda\in\KK$ and let $V(f-\lambda)=\lbrace P\in\KK^2 \mid (f-\lambda)(P)=0\rbrace$ be the curve of $\KK^2$ defined by $f-\lambda$. Given $P=(a,b)\in V(f-\lambda)$, we denote by $\mu_P^{\lambda}$ the \emph{local Milnor number} of $(f-\lambda)$ at $P$ (if $\mathfrak{m}_P=(X-a,Y-b)$, then $\mu_P^{\lambda}$ is defined to be the rank of the $\KK$-vector space $\KK[X,Y]_{\mathfrak{m}_P}/{(f,g)\KK[X,Y]_{\mathfrak{m}_P}}$). We say that $f-\lambda$ is \emph{singular} at $P$ if $\mu_P^{\lambda}>0$, otherwise, $P$ is a \emph{smooth point} of $f-\lambda$. We say that $f-\lambda$ is \emph{singular} if $f-\lambda$ has at least one singular point. In our setting, if $f-\lambda$ is singular, then it has a finite number of singular points. Furthermore, there is a finite number of translates of $f$ which are singular. Let $\mu(f)=\dim_{\KK}\frac{\KK[X,Y]}{(f_X,f_Y)}$, then 
\[\mu(f)=\sum_{\lambda\in{\bf K}}\sum_{P\in V(f-\lambda)}\mu_P^{\lambda},\] 
that is, $\mu(f)$ is the sum of local Milnor numbers at the singular points of the translates of $f$.

Write \[X(t)=t^n+\alpha_1t^{n-1}+\dots+\alpha_n\] and \[Y(t)=t^m+\beta_1t^{m-1}+\dots+\beta_m\] and suppose, without loss of generality, that $m<n$ and also (by taking the change of variables $t_1=t+\frac{\beta_1}{m}$) that $\beta_1=0$. We can express $f(X,Y)$ as $f(X,Y)=Y^n+a_1(X)Y^{n-1}+\dots+a_n(X)$. Clearly $n,m\in \Gamma(f)$. Let $d$ be a divisor of $n$ and let $g(X,Y)$ be a $Y$-monic polynomial of degree $\frac{n}{d}$ in $Y$. Let 
\[
f=g^{d}+c_1(X,Y)g^{d-1}+\dots+c_d(X,Y)
\]  
with deg$_Yc_i(X,Y)<\frac{n}{d}$ for all $i\in\lbrace 1,\ldots,d\rbrace$, the expansion of $f$ with respect to  $g$. We say that $g$ is a \emph{$d$th approximate root} of $f$ if $c_1(X,Y)=0$. It is well known that a $d$th approximate root of $f$ exists and it is unique. We denote it by 
$\App(f,d)$. With these notations we have the following algorithm that computes a set of generators of $\Gamma(f)$ (see for instance \cite{ag1}).

\begin{alg} \label{alg-di}
Let $r_0=m=d_1$ and let $r_1=n$. Let $d_2=\gcd(r_0,r_1)=\gcd(r_1,d_1)$ and let $g_2=\App(f,d_2)$. We set $r_2=\mathrm{d}(g_2(X(t),Y(t)))$ and $d_3=\gcd(r_2,d_2)$ and so on.  

It follows from \cite{ab} that there exists $h>1$ such that $d_{h+1}=1$, and also that $\Gamma(f)=\langle r_0,r_1,\ldots,r_h\rangle$. We set $e_k=\dfrac{d_k}{d_{k+1}}$ for all $k\in\lbrace 1,\ldots,h\rbrace$.
\end{alg}

The following Proposition gives the main properties of $\Gamma(f)$.

\begin{proposicion} Let $f$, $r_i$, $d_i$, and $e_i$ be defined as above. We have the following:

\begin{enumerate}
    
\item $\Gamma(f)$ is free with respect to the arrangement $(r_0,r_1,\ldots,r_h)$.

\item $r_kd_k>r_{k+1}d_{k+1}$ for all $k\in\lbrace 1,\ldots,h\rbrace$.

\item $\mathrm{d}(f_y(x(t),y(t))=\sum_{i=1}^h(e_i-1)r_i$.

\item $\mathrm C(\Gamma(f))=\mu(f)=\mathrm{d}(f_y(x(t),y(t)))-n+1$.
\end{enumerate}
\end{proposicion}

\begin{proof} See \cite{ag}.
\end{proof}

Let $\BB=\frac{\KK[X,Y]}{(f)}$ and let $x,y$ be the images of $X,Y$ in $\BB$. Let $\mathbf N=\BB  dx+\BB dy$ be the $\BB $-module generated by $\{dx, dy\}$, and let $\bar\BB$ be the integral closure of $\BB$.  Let $\tilde{\mathbf N}=\bar{\BB}dx+\bar{\BB}dy$. Let $\nu(f)=\dim_{\KK}\frac{\KK[X,Y]}{(f,f_X,f_Y)}={\dim}_{\KK}\frac{\BB}{(f_X,f_Y)}$. If we denote by $\ell(\cdot)$ the length of the module, then we have the following property.

\begin{proposicion}\cite[Corollary 2]{b} Let $f$ be defined as above.
\[\nu(f)=\ell\left(\tilde{\bf N}/{\bf N}\right)+\frac{\mu(f)}{2}.\]
\end{proposicion}

In our setting, $\BB\simeq \KK[X(t),Y(t)]={\bf A}$, hence $\bar{\bf B}\simeq \bar{\bf A}=\KK[t]$, where $\bar{\bf A}$ is the integral closure of  ${\bf A}$. It follows that $\bf N$ is isomorphic to ${\bf M}=x'(t){\bf A}+y'(t){\bf A}$ and also that $\tilde{\bf N}$ is isomorphic to $\tilde{\bf M}=x'(t)\KK[t]+y'(t)\KK[t]=\lbrace g'(t), g(t)\in\KK[t]\rbrace$.

Note that if $g(X,Y)\in\KK[X,Y]$, then $\frac{d}{dt}g(X(t),Y(t))\in {\bf M}$,  whence $\mathrm{d}(\frac{d}{dt}g(X(t),Y(t)))\in I=\mathrm{d}({\bf M})$. It follows that $\lbrace s-1 \mid s\in \Gamma(f)\rbrace\subseteq I$ and $\mathrm{d}(\frac{d}{dt}g(X(t),Y(t)))$ is an exact element. In particular,  $\ell(\frac{\tilde{\bf N}}{\bf N})$ is the cardinality of the set $\lbrace s\in \mathrm G(\Gamma(f))\mid s-1\notin S\rbrace$. This cardinality is nothing but $\mathrm g(\Gamma(f))-\mathrm{ne}(\mathbf M)=\frac{\mu(f)}{2}-\mathrm{ne}(\mathbf M)$, and it follows that:

\begin{enumerate}

\item $\nu(f)=\mu(f)=\mathrm C(\Gamma(f))$ if and only if $\mathrm{ne}(\mathbf M)=0$, that is, every element of $I$ is exact;

\item $\nu(f)=\frac{\mu(f)}{2}$ if and only if $\mathrm{ne}(\mathbf M)=\mathrm g(\Gamma(f))$.
\end{enumerate}

In the following, we shall introduce the notion of characteristic exponents of $f$. Then we shall prove that, after possibly a change of variables, the curve $\mathrm V(f)$ has a parametrization in one of the following forms:

\begin{enumerate}

\item $X(\tau)=\tau^n, Y(\tau)=\tau^m$ (hence the equation of the curve is of the form $W^n-Z^m$),

\item $X(\tau)=\tau^n+c_{\lambda}t^{\lambda}+\ldots, Y(\tau)=\tau^n$ and $m+\lambda\notin \Gamma(f)$ (hence the degree of  $mX'(\tau)Y(\tau)-nX(\tau)Y'(\tau)$ is a non exact element of $I$).
\end{enumerate}
We will need to this end this technical Lemma.

\begin{lema}\label{isom} Let $q(t)=t+\sum_{i\geq 1} c_it^{-i}\in \KK(\!(t)\!)$ and define the map $l:\KK(\!(T)\!)\to \KK(\!(t)\!)$, $\alpha(T)\mapsto \alpha(q(t))$. In particular, $l(T)=q(t)$. Then $l$ is an isomorphism.
\end{lema}

\begin{proof} We clearly have $l(\alpha(T)+\beta(T))=l(\alpha(T))+l(\beta(T))$ and $l(\alpha(T)\beta(T))=l(\alpha(T))l(\beta(T))$ for all $\alpha(T),\beta(T)\in\KK(\!(T)\!)$. Furthermore, $l(1)=1$ and $\ker(l)=\lbrace 0\rbrace$. We shall now construct the inverse of $l$ by proving that $t=T+b_1T^{-1}+b_2T^{-2}+\ldots$ for some $b_i\in \mathbf K$. We shall do this by induction on $k\geq 1$. More precisely we shall prove that for all $k\geq 1$, there exist $b_k\in \KK$ such that  deg$_t (t-l(T+b_1T^{-1}+\dots+b_kT^{-k}))\leq -k-1$. We shall use the fact that for all $k\in \mathbb{Z}$, we can write
$$
l((q(t))^k)=t^k+\sum_{i\geq 1}c_{i}^{(k)}t^{k-i-1}
$$
for some $c_i^{(k)}\in{\bf K}$. If $k=1$, then we set $b_1=-c_1$. We have $t-l(T)-b_1l(T^{-1})=(-c_1-b_1)T^{-1}-\sum_{i\geq 2}c^{(1)}_it^{-i}=\sum_{i\geq 2}-c^{(1)}_it^{-i}$. Hence the assertion is clear. Suppose that the assertion is true for $k$ and let us prove it for $k+1$. By hypothesis we have 
$$
t=l(T+b_1T^{-1}+\dots+b_kT^{-k}) +\sum_{i\geq 1}c_i^{(k)}t^{-k-i}.
$$
Then we set $b_{k+1}=c_1^{(k)}$. But $l(c_1^{(k)}(q(t))^{-k-1})=c_1t^{-k-1}+\sum_{i\geq 1}\bar{c}_i^{(k+1)}t^{-k-2-i}$. Hence $t=l(T+b_1T^{-1}+\ldots+b_{k+1}T^{-k-1})+\sum_{i\geq 1}c_i^{(k+1)}t^{-k-1-i}$. This proves the assertion for $k+1$.

Let $q_1(T)=T+\sum_{k\geq 1}b_kT^{-k}$ and set $l_1(\gamma(t))=\gamma(q_1(T))$ (in particular $l_1(t)=q_1(T)$). Since deg$_t(t-l(q_1(T))\leq -k$ for all $k\geq 0$, then $t=l(q_1(T))$. This proves that $l$ is surjective, hence an ismorphism. Note that $l_1=l^{-1}$ because $l(l_1(t))=t$.
\end{proof}

Let us make the following change of variables 
\[T=t(1+\beta_2t^{-2}+\dots+\beta_mt^{-m})^{\frac{1}{m}}=t\left(1+\frac{1}{m}\beta_2t^{-2}+\ldots\right)=q(t).\] 

This change of variables defines a map $l:\KK(\!(T)\!)\to \KK(\!(t)\!)$, $l(T)=q(t)$. It follows from Lemma \ref{isom} that $l$ is an isomorphism. Let  $X_1(T)=X(l^{-1}(t))$ and $Y_1(T)=Y(l^{-1}(t))$. We have  \[Y_1(T)=T^m\hbox{ and }X_1(T)=T^n+\sum_{p< n}c_pT^{p},\] 
for some $c_p\in \mathbf K$, and we can easily verify that for all $g(X,Y)\in\KK[X,Y]$, $\mathrm{d}(g(X(t),Y(t)))$ is also the degree in $T$ of $g(X_1(T),Y_1(T))$. Furthermore, $\frac{d}{dt}(g(X(t),Y(t)))=\frac{d}{dT}(g(X_1(T),Y_1(T)))\frac{dT}{dt}$. 

Recall that the Newton-Puiseux exponents of $f$ are defined as follows: let $m_1=-n$ and let $D_2={\gcd}(m,n)=d_2$. For all $i\geq 2$ define $-m_i={\max}\lbrace p \mid  D_i\nmid p\rbrace$  and $D_{i+1}={\gcd}(D_i,m_i)$.  We have $D_{h+1}=d_{h+1}=11$ and $D_i=d_i$ for all $i\in\lbrace 1,\ldots,h\rbrace$ (the $d_i$ where defined in Algorithm \ref{alg-di}).

The Newton-Puiseux exponents are related to the sequence $r_0,\ldots,r_h$ by the following relation: $r_0=m$, $r_1=n$, and for all $k\geq 1, -r_{k+1}=-r_ke_k+(m_{k+1}-m_k)$ where we recall that  $e_k=\dfrac{d_k}{d_{k+1}}$ for all $k\in\lbrace 1,\ldots,h\rbrace$. In particular, $r_2=r_1e_1+m_2-m_1=r_1e_1-m_2-r_1=(e_1-1)r_1-m_2$. Hence $-m_2<r_2$.
 
Let  $\lambda=\max\lbrace p \mid p<n, c_p\not=0\rbrace$ and suppose that $\lambda>-\infty$. We have: 
$$
X_1(T)=T^n+c_{\lambda}T^{\lambda}+\dots \mbox{ and } Y_1(T)=T^m.
$$
The hypothesis on $\lambda$ implies that $c_{\lambda}\not=0$. Let
$$
W(T)=mX_1'(T)Y_1(T)-nY_1'(T)X_1(T).
$$
We have $W(T)=(m\lambda-nm)c_{\lambda}T^{m+\lambda-1}+\dots$. If $m+\lambda\notin \Gamma(f)$, then $m+\lambda-1$ is a non exact degree.

Suppose that $m+\lambda\in \Gamma(f)$.  We have then the following two possibilities.
 
 \begin{enumerate}
 \item $\lambda>-m_2$. In this case, $d_2\mid\lambda$. Hence $\lambda$ is in the group generated by $n,m$. Then 
 $m+\lambda=an+bm$ for some $a,b\in{\mathbb N}$.
 
 \item   $\lambda=-m_2$. In this case, $m+\lambda=m-m_2=an+bm+cr_2$ for some $a,b,c\in{\mathbb N}, c\not=0$. But 
 $m-m_2=m+r_2-(e_1-1)r_1$. Thus, $m+r_2-(e_1-1)r_1=an+bm+cr_2$. If $c\geq 1$, then $m-(e_1-1)r_1=an+bm+(c-1)r_2$, which is a contradiction because $m-(e_1-1)r_1=m-(e_1-1)n<0$. It follows that $c=0$, whence $m+r_2-(e_1-1)r_1=an+bm$, and $r_2=(a+e_1-1)n+(b-1)m$, but $d_2=\gcd(n,m)$ does not divide $r_2$. This is again a contradiction.
 \end{enumerate}
It follows that $\lambda<-m_2$ and $m+\lambda=an+bm$ for some $a,b\in{\mathbb N}$. Since $n>m>\lambda$ then $a\leq 1$. Furthermore, if $a=1$, then $b=0$.  Hence one of the following conditions holds.
 
\begin{enumerate}
\item $m+\lambda=n$. Let in this case $Y_2=Y_1+\alpha, \alpha\in\KK^*$. We have
$$
\bar{W}(T)=mX_1'(T)Y_2(T)-nY_2'(T)X_1(T)=[(m\lambda-nm)c_{\lambda}-\alpha mn]T^{n-1}+\cdots
$$
Hence, if $\alpha=\frac{\lambda-n}{n}c_{\lambda}=-\frac{m}{n}c_{\lambda}$, then $\bar{W}(T)$ has degree strictly less than $n-1$. As an example of this case, let $X(t)=t^9+t^5, Y(t)=t^4$. We have $W(t)=16t^8$ and $8+1=9\in \mathrm{d}({\bf A})$. If $\bar{Y}=t^4+\frac{4}{9}$, then 
 $\bar{W}(t)=mX'(t)\bar{Y}(t)-n\bar{Y}'(t)X(t)=\frac{-80}{9}t^4$ and $4+1\notin \mathrm{d}({\bf A})$.
 
 \item $m+\lambda=\theta m$. In this case, $\lambda=(\theta-1)m$. The change of variables $X_2=X_1-Y_1^{\theta-1}, Y_2=Y_1$ is such that either $(X_2,Y_2)=(T^n, T^m)$ or $X_2=T^n+c_{\lambda_1}T^{\lambda_1}+\ldots, Y_2=T^m$ with $\lambda_1<\lambda$. As an example of this case, let $X(t)=t^{7}, Y(t)=t^4+t$. We have $W(t)=-21t^7$ and $7+1=8=2.4\in \mathrm{d}({\bf A})$. Let $Y_1=T^4$. Then $T^4=t^4+t$, $T=t(t^{-3}+1)^{\frac{1}{4}}$, and $X_1(T)=T^7-\frac{1}{4}T^4+\frac{7}{16}T+\ldots$. If $X_2=X_1+\frac{1}{4}Y_1, Y_2=Y_1$, then 
 $X_2=T^7+\frac{7}{16}T+\cdots, Y_2=T^4$ and $mX_2'(T)Y_2(T)-nY_2'(T)X_2(T)=\frac{21}{2}T^4+\ldots$, with $4+1=5\notin \mathrm{d}({\bf A})$.
  
 \end{enumerate}
We shall prove that these two processes will eventually stop. This is clear for the first case since we are constructing a strictly increasing sequence  of nonnegative integers. In the second case, if $h\geq 2$ then this is clear since the set of integers in the interval $[\lambda, -m_2]$ is finite. Suppose that $h=1$, that is, $\gcd(m,n)=1$. If the process is infinite, then after a finite number of steps we will obtain a new parametrization of the curve of the form $\tilde{X}=T^n+\alpha T^{-l}+\ldots, \tilde{Y}=T^m$ with $l>nm$, which is a contradiction. 
 
It follows that either we get a parametrization $(\tau^n,\tau^m)$ of the curve $V(f)$ (which means that the equation of this curve is $W^n-Z^m$ with $\KK[X,Y]\simeq\KK[Z,W]$ and $\gcd(n,m)=1$), or
we get a new parametrization $Z(t)=t^n+a_1t^{\alpha_1}+\dots+a_n, W(t)=t^m+b_1t^{\beta_1}+\dots+b_m$ such that the degree of $W(t)=mZ'(t)W(t)-nW'(t)Z(t)$ is a non exact element of $I$.

We then get the follwong result.
 
\begin{teorema} (see also [2]) Let $X(t)=t^n+a_1t^{n-1}+\dots+a_n$, $Y(t)=t^m+b_1t^{m-1}+\dots+b_m$ be the equations of a polynomial curve in $\KK^2$ and let $f(X,Y)$ be the minimal polynomial of $X(t),Y(t)$, that is, $f(X,Y)$ is the resultant in $t$ of $(X-X(t),Y-Y(t))$.  Let ${\bf M}=X'(t){\bf A}+Y'(t){\bf A}$ be the ${\bf A}$-module generated by $X'(t),Y'(t)$. The following conditions are equivalent.
\begin{enumerate}[i)]
\item The equality $\mu(f)=\nu(f)$ holds.
 
\item Every element of the ideal $I=\mathrm{d}({\bf M})$ is exact.
 
\item The integers $n$ and $m$ are coprime and there exist an isomorphism $\KK[X,Y]\to \KK[Z,W]$ such that the image of $f(X,Y)$ is $W^n-Z^m$.
\end{enumerate} 
\end{teorema}
\begin{proof} i) $\Longleftrightarrow$ ii) is clear and ii) $\Longrightarrow$ iii) results from the calculations above. Finally iii) $\Longrightarrow$ i) because $W^n-Z^m\in (W^{n-1},Z^{m-1})$.
 \end{proof}
Let the notations be as above and let $W(t)=mX'(t)Y(t)-nY'(t)X(t)$. 

If $W(t)=0$, then $mX'(t)Y(t)=nY'(t)X(t)$. Hence $Y(t)^n-X(t)^m=0$. In particular, $f(X,Y)=Y^n-X^m$. 

If $W(t)\not=0$ and $W(t)$ is exact, then similar calculations as above show that there exists a change of variables in such a way that the new $W$ is either $0$ or its degree is a  non exact element. Assume that $f(X,Y)$ is not equivalent to a quasi-homogeneous polynomial, in particular we may assume that $W(t)$ is not exact. In the following we shall give a bound for the number of non exact elements of $I$.


\begin{proposicion} \label{bound}Let the notations be as above. If $\mathrm{ne}(\mathbf M)>0$ then  $\mathrm{ne}(\mathbf M)\geq 2^{h-1}$.
\end{proposicion} 
\begin{proof} Consider as above the parametrization $X(T)=T^n+c_{\lambda}T^{\lambda}+\dots, Y(T)=T^m$ and let $\mathrm{d}(W)=m+\lambda-1$. We have $m+\lambda\notin \Gamma(f)$. Furthermore, $\lambda\geq -m_2$. Let $g_i(X,Y)={\mathrm App}(f,d_i)$ for all $i\in\lbrace 1,\ldots,h\rbrace$. We have two cases.

\begin{enumerate}

\item $\lambda>-m_2$. We have $m+\lambda=-am+bn$ with $a,b\in{\mathbb N}, a>0, 0\leq b\leq e_1$. Hence, for all $(\alpha_2,\ldots,\alpha_h)\in {\mathbb N}^{h-1}$, if $\alpha_i<e_i$, then for every $i\in\lbrace 2,\ldots,h\rbrace$, the degree of $g_2^{\alpha_2}\cdots g_h^{\alpha_h}W$ is not exact, hence $\mathrm{ne}({\mathbf M})\geq 2^{h-1}$.

 \item  $\lambda=-m_2$. We have $m+\lambda=m-m_2=-am+bn+cr_2$ with $a,b,c\in{\mathbb N}$, $a>0$, $0\leq b<e_1$, $0\leq c<e_2$. But $-m_2=r_2-(e_1-1)r_1$. Thus $m+r_2-(e_1-1)r_1=-am+bn+cr_2$, and since $(c-1)r_2$ is not divisible by $d_2$, we get $c=1$, whence $(a+1)m=(e_1-1+b)r_1=(e_1-1+b)n$. If $b=0$, then $(e_1-1)n$ is divisible by $m$, which is a contradiction. Hence $e-1-1+b\geq 2$, which implies that $a\geq 2$. Finally $m+\lambda=-am+bn+r_2$ with $a\geq 2$. Note that 
 $\mathrm{d}(YW)=-(a-1)m+bn+r_2$, and thus $\mathrm{d}(YW)$ is not exact. Furthermore, for all $(\alpha_3,\ldots,\alpha_h)\in {\mathbb N}^{h-2}$, if $\alpha_i<e_i$ for all $i\in \lbrace 3,\ldots,h\rbrace$, then the degree of $Yg_3^{\alpha_3}\cdots g_h^{\alpha_h}W$ is not exact. It follows that 
 $\mathrm{ne}(\mathbf M)\geq 2^{h-1}$. \qedhere
 \end{enumerate}
\end{proof}
 
\begin{corolario}\label{non exact} With the notations above. We have the following.
\begin{enumerate}[(1)]
 \item If $\mathrm{ne}(\mathbf M)=1$, then $h=1$, that is, $S=\langle m,n\rangle $ with $\gcd(m,n)=1$. Furthermore, $\mathrm{NE}({\bf M})=\lbrace \mathrm{F}(\Gamma(f))-1\rbrace$.
 
\item If $\mathrm{ne}(\mathbf M)=2$, then $h\in\{1,2\}$, that is,  either $\Gamma(f)=\langle m,n\rangle$ with $\gcd(m,n)=1$ or $\Gamma(f)=\langle m,n,r_2\rangle$ with  $d_3=1$. Furthermore, if $h=1$ (respectively $h=2$), then $\mathrm{NE}({\bf M})$ is either $\lbrace \mathrm{F}(\Gamma(f))-1, \mathrm{F}(\Gamma(f))-m-1\rbrace$ or $\lbrace \mathrm{F}(\Gamma(f))-1, \mathrm{F}(\Gamma(f))-n-1\rbrace$ (respectively $\mathrm{NE}({\bf M})$ is either $\lbrace \mathrm{F}(\Gamma(f))-1, \mathrm{F}(\Gamma(f))-n-1\rbrace$ or $\lbrace \mathrm{F}(\Gamma(f))-1, \mathrm{F}(\Gamma(f))-m-1\rbrace$ or $\lbrace \mathrm{F}(\Gamma(f))-1, \mathrm{F}(\Gamma(f))-r_2-1\rbrace$).
\end{enumerate} 
\end{corolario}
\begin{proof} 
(1) The first assertion results from Proposition \ref{bound}, and obviously $\mathrm{NE}({\bf M})=\lbrace m+\lambda-1\rbrace$. If $m+\lambda<\mathrm{F}(\Gamma(f))$ then $m+\lambda=-am+bn$ with  $a\geq 1$ and $b\leq m-1$. If $a>1$ then $XW$ is not exact and $XW\not=W$. This is a contradiction. If $a=1$ then $b<m-1$ (otherwise $m+\lambda=-a+(m-1)n=\mathrm{F}(\Gamma(f))$ which contradicts the hypothesis). But $YW$ is not exact and $YW\not=W$. This is again a contradiction. 

(2) The first assertion results from Proposition \ref{bound}. To prove the second assertion, let $W_1$ have a non exact degree with $\mathrm{d}(W_1)<\mathrm{F}(\Gamma(f))-1$ and $\mathrm{d}(W_1)$ is minimal in $\mathrm{ne}(\mathbf M)$. Suppose first that $h=1$. We have $\mathrm{d}(W_1)+1=-am+bn$ with $a\geq 1$ and $0\leq b\leq m-1$. If $a\geq 2$ and $b<m-1$, then $XW_1,YW_1$ have also non exact degrees, and thus $\mathrm{ne}(\mathbf M)\geq 3$, which is a contradiction. Consequently  either $a=1$ or $b=m-1$. If $a=1$, then $b<m-1$, whence $YW_1,\ldots,Y^{m-1-b}W_1$ have also non exact degrees. This forces $b$ to be equal to $m-2$. Consequently $\mathrm{NE}({\bf M})=\lbrace \mathrm{F}(\Gamma(f))-1,\mathrm{F}(\Gamma(f))-1-n\rbrace$. If $b=m-1$, then we prove in a similar way that $\mathrm{NE}({\bf M})=\lbrace \mathrm{F}(\Gamma(f))-1,\mathrm{F}(\Gamma(f))-1-m\rbrace$.
 
Suppose now that $h=2$. Then $\mathrm{d}(W_1)+1=-an+bm+cr_2$ with $a\geq 1$, $0\leq b\leq e_1-1$, $0\leq c\leq e_2-1$, and $(a,b,c)\not=(-1,e_1-1,e_2-1)$ (otherwise $\mathrm{d}(W_1)+1=\mathrm{d}(\mathrm{F}(\Gamma(f)))$. This forces $\mathrm{ne}(\mathbf M)$ to be equal to $1$ by the minimality of $\mathrm{d}(W_1)$). On the other hand, $\mathrm{ne}(\mathbf M)=2$ forces  $(a,b,c)$ to be either $(-1,e_1-2,e_2-1)$ (and thus $\mathrm{NE}({\bf M})=\lbrace \mathrm{F}(\Gamma(f))-1, \mathrm{F}(\Gamma(f))-m-1\rbrace$) or $(-2,e_1-1,e_2-1)$ (hence $\mathrm{NE}({\bf M})=\lbrace \mathrm{F}(\Gamma(f))-1, \mathrm{F}(\Gamma(f))-n-1\rbrace$) or $(-1,e_1-1,e_2-2)$ (in this case $\mathrm{NE}({\bf M})=\lbrace \mathrm{F}(\Gamma(f))-1, \mathrm{F}(\Gamma(f))-r_2-1\rbrace$). 
 \end{proof}
 
In the following we shall give more precise information when $\mathrm{ne}(\mathbf M)\in\{1,2\}$.
 
\noindent{\bf The case of one non exact element.} In this case $h=1$, $\Gamma(f)=\langle m,n\rangle$ with $m<n$ and $\gcd(m,n)=1$. Furthermore, $m+\lambda=\mathrm{F}(\Gamma(f))=-m+(m-1)n<m+n$ because $\lambda <n$. This implies that $(m-2)n<2m<2n$. In particular, $m<4$. If $m=2$, then $n=2p+1$ for some $p\geq 1$. If $m=3$, then $n<2m=6$ and $n>m=3$ implies that either $n=4$ or $n=5$. 
 
\noindent {\bf The case of two non exact elements and $h=1$.} In this case, $\Gamma(f)=\langle m,n\rangle$ with $m<n$ and $\gcd(m,n)=1$. Furthermore, By Corollary \ref{non exact},  $m+\lambda\in\lbrace  \mathrm{F}(\Gamma(f))-n, \mathrm{F}(\Gamma(f))-m\rbrace$.
 
 \begin{enumerate}
 
 
 \item If $m+\lambda=\mathrm{F}(\Gamma(f))-m=-2m+(m-1)n$, then we get, using the fact that $\lambda<n$, $6>(n-3)(m-2)$. Hence $(m,n)$ is either  $(2,2p+1), p\geq 1$ or $(3,4)$ or $(3,5)$ or $(4,5)$.
 
 \item If $m+\lambda=\mathrm{F}(\Gamma(f))-n=-m+(m-2)n$, then we get that $4>(n-2)(m-3)$. In particular, $(m,n)$ is either $(2,2p+1)$ with $p\geq 1$, or $(3,n)$ with $n\geq 4$ and $\gcd(3,n)=1$, or $(4,5)$.
 
 \end{enumerate}
 
 
 
 
 
 \medskip
 
 \noindent {\bf The case of two non exact elements and $h=2$}. Let $\Gamma(f)=\langle m,n,r_2\rangle$ and let the notations be as above. Since $\mathrm{F}(\Gamma(f))-1$ is a non exact element of $I$, we have $m+\lambda\in\lbrace \mathrm{F}(\Gamma(f)), \mathrm{F}(\Gamma(f))-m,\mathrm{F}(\Gamma(f))-n, \mathrm{F}(\Gamma(f))-r_2\rbrace$. 
 
 \begin{enumerate}
 
 \item If $m+\lambda=\mathrm{F}(\Gamma(f))=-m+(e_1-1)n+(e_2-1)r_2$, then $\lambda=-m_2=r_2-(e_1-1)n$ (because otherwise $\lambda > -m_2$, whence $d_2$ divides $\lambda$ and consequently $m+\lambda$ is in $d_2{\mathbb Z}$, which is a contradiction). This implies that  $m+r_2-(e_1-1)n=-m+(e_1-1)n+(e_2-1)r_2$. Since $d_2=\gcd(m,n)$ does not divide  $ir_2$ for all $1\leq i\leq e_2-1$, we deduce that $e_2=2$. This implies that  $m-(e_1-1)n=0$, which is a contradiction since $m<n$.
  
 \item  Suppose that $m+\lambda=\mathrm{F}(\Gamma(f))-r_2=-m+(e_1-1)n+(e_2-2)r_2$. If $e_2\not= 2$, then by the same argument as in (1), $\lambda=-m_2=r_2-(e_1-1)n$. Hence $m+r_2-(e_1-1)n=-m+(e_1-1)n+(e_2-2)r_2$. Since $d_2$ does not divide $ir_2$ for all $1\leq i\leq e_2-1$, we obtain $e_2=3$, but $m-(e_1-1)n=0$, which is a contradiction. It follows that $e_2=2$, whence $d_2=2$ and $\lambda>-m_2$ (because $\lambda=-2m+(e_1-1)n$ and $m_2$ is not divisible by $d_2$). But $m+\lambda=\mathrm{F}(\Gamma(f))-r_2=-m+(e_1-1)n<m+n$, and thus  
 $$
 -2m+\left(\frac{m}{2}-1\right)n<n.
 $$
Let $a=\frac{m}{2}$ and $b=\frac{n}{2}$. Since $\gcd(m,n)=2$, $a$ and $b$ are coprime. The equality above implies that $2a>(a-2)b$. But $a<b$. Hence $2b>(a-2)b$, that is, $a<4$. Note that $a>1$ because $h=2$. If $a=2$, then $b=2p+1$ for some $p\in{\mathbb N}$. If $a=3$, then $b<6$, and consequently $b\in\{4,5\}$. This implies that $(m,n,r_2)$ satisfies one of the following conditions:
 
\begin{enumerate}
 \item $m=4,n=4p+2, r_2=2q+1$ and $8p+4>2k+1$.
 
 \item $m=6,n=8, r_2=2p+1$ and $24>2p+1$.
 
 \item $m=6,n=10, r_2=2p+1$ and $30>2p+1$.

\end{enumerate}
 
\item  If $m+\lambda=\mathrm{F}(\Gamma(f))-m=-2m+(e_1-1)r_1)+(e_2-1)r_2$, then $\lambda=-m_2=r_2-(e_1-1)n$, which implies that 
$m+r_2-(e_1-1)n=-2m+(e_1-1)r_1)+(e_2-1)r_2$. Hence $e_2=d_2=2$ and $3m=2(e_1-1)n=2(\frac{m}{2}-1)n$. Consequently $3(\frac{m}{2}-1)+3=(\frac{m}{2}-1)n$. Finally, $(n-3)(\frac{m}{2}-1)=3$. The only solution is $m=4$, $n=6$. Hence $r_2=2p+1$ with $12>r_2$.
 
\item If $m+\lambda=\mathrm{F}(\Gamma(f))-n=-m+(e_1-2)r_1)+(e_2-1)r_2$ then $\lambda=-m_2=r_2-(e_1-1)n$, which implies that $m+r_2-(e_1-1)n=-m+(e_1-2)r_1+(e_2-1)r_2$. Thus $e_2=d_2=2$ and $2m=(2e_1-3)n=(m-3)n$. This yields  $6=(m-3)(n-2)$. All possible cases lead to a contradiction.
\end{enumerate}
 
These results can be summerized into the following theorem.
 
 \begin{teorema}  Let $X(t)=t^n+a_1t^{n-1}+\dots+a_n$, $Y(t)=t^m+b_1t^{m-1}+\dots+b_m$, and assume that $m<n$ and $\gcd(m,n)<m$. Let $f(X,Y)$ be the monic irreducible polynomial of $\KK[X,Y]$ such that $f(X(t),Y(t))=0$, and let $\Gamma(f)$ be the semigroup associated with $f$. Assume that $\Gamma(f)$ is a numerical semigroup and let $\Gamma(f)=\langle m=r_0,n=r_1.r_2,\ldots,r_h\rangle$ where $r_2,\ldots,r_h$ are constructed as in Algorithm \ref{alg-di}. Let $\mu(f)=\dim_{\KK}{\KK[X,Y]}/{(f_X,f_Y)}$ and $\nu(f)=\dim_{\KK}\KK[X,Y]/{(f,f_X,f_Y)}$. Assume that $\mu(f)>\nu(f)$. 
\begin{enumerate}[i)]
\item  If $\mu(f)=\nu(f)+1$, then $h=1$.
\item If $\mu(f)=\nu(f)+2$, then $h\in\{1,2\}$.
\end{enumerate} 
Moreover, the following also holds.
\begin{enumerate}[(1)]
\item If $\mu(f)=\nu(f)+2$, then $\Gamma(f)=\langle m,n\rangle$ and one of the following conditions holds:
 
\begin{itemize}
\item $(m,n)=(2,2p+1), p\geq 1$,
\item $(m,n)=(3,4)$,
\item $(m,n)=(3,5)$.
\end{itemize}

\item If $\mu(f)=\nu(f)+2$ and  $h=1$, then  $\Gamma(f)=\langle m,n\rangle$ and one of the following conditions holds:
 
\begin{itemize}
\item $(m,n)=(2,2p+1), p\geq 1$,
\item $(m,n)=(3,4)$,
\item $(m,n)=(3,5)$,
\item $(m,n)=(4,5)$,
\item $(m,n)=(3,n)$, with $\gcd(3,n)=1$.
\end{itemize}

\item If $\mu(f)=\nu(f)+1$ and $h=2$m then  $\Gamma(f)=\langle m,n,r_2\rangle$ and one of the following conditions holds:
 \begin{itemize}
\item $(m,n,r_2)=(4,4p+2, 2q+1)$, $p\geq 1$ and $8p+4>2q+1$,
\item $(m,n,r_2)=(6,8,2p+1)$, $p\leq 11$,
\item $(m,n,r_2)=(6,10,2p+1)$, $p\leq 14$,
\item $(m,n,r_2)=(4,6,2p+1)$, $p\leq 5$.
\end{itemize}
\end{enumerate}
\end{teorema}

  \end{document}